\documentclass[11pt,a4paper]{amsart}
\usepackage{amsmath, amssymb}
\usepackage{amsfonts}
\usepackage[arrow,matrix,curve,cmtip,ps]{xy}
\usepackage{amsthm}
\usepackage{amscd}
\usepackage{hyperref}

\allowdisplaybreaks

\theoremstyle{remark}
\newtheorem{theorem}{{Theorem}}[subsection]
\newtheorem{lemma}[theorem]{Lemma}
\newtheorem{proposition}[theorem]{Proposition}
\newtheorem{corollary}[theorem]{Corollary}

\newtheorem{remark}[theorem]{Remark}
\newtheorem{definition}[theorem]{Definition}

\newtheorem{question}{Question}

\numberwithin{equation}{subsection}
\numberwithin{theorem}{subsection}
\newcount\colveccount
\newcommand*\colvec[1]{
        \global\colveccount#1
        \begin{pmatrix}
        \colvecnext
}
\def\colvecnext#1{
        #1
        \global\advance\colveccount-1
        \ifnum\colveccount>0
                \\
                \expandafter\colvecnext
        \else
                \end{pmatrix}
        \fi
}

\newcommand{\M}{\mathsf{M}}

\newcommand{\R}{\mathbb{R}}

\newcommand{\defeq}{\mathrel{\mathop:}=}

\begin{document}
\title[Anosov Structure on Margulis Space Time]{Anosov Structure on Margulis Space Time}

\author{Sourav Ghosh}

\address{Department of Mathematics \\ University of Paris 11 \\ Orsay 91400 \\ France}
\email{sourav.ghosh@math.u-psud.fr}

\thanks{The research leading to these results has received funding from the European Research Council under the {\em European Community}'s seventh Framework Programme (FP7/2007-2013)/ERC {\em  grant agreement}}

\date{\today}



\begin{abstract}
In this paper we describe the stable and unstable leaves for the geodesic flow on the space of non-wandering geodesics of a Margulis Space Time and prove contraction properties of the leaves under the flow. We also show that mondromy of Margulis Space Times are ``Anosov representations in non semi-simple Lie groups".

\end{abstract}

\maketitle
\tableofcontents
\pagebreak

\section{Introduction}

A Margulis Space Time $\mathsf{M}$ is a quotient manifold of the three dimensional affine space by a free, non-abelian group acting as affine transformations with discrete linear part. It owes its name to Grigory Margulis, who was the first to use these spaces, in \cite{marg1} and \cite{marg2}, as examples to answer Milnor's following question in the negative.
\begin{question}
Is the fundamental group of a complete, flat, affine manifold virtually polycyclic? \cite{milnor}
\end{question}
Observe that if $\mathsf{M}$ is a Margulis Space Time then the fundamental group $\pi_1(\mathsf{M})$ does not contain any translation. By combining results of Fried, Goldman and Mess from \cite{fried}, \cite{mess}, a complete flat affine manifold either has a polycyclic fundamental group or is a Margulis Space Time. In this paper we will only consider Margulis Space Times whose linear part contains no parabolic, although by Drumm there exists Margulis Space Time whose linear part contains parabolics. Fried and Goldman showed in \cite{fried} that a conjugate of the linear part of the affine action of the fundamental group forms a subgroup of $\mathsf{SO}^0(2,1)$ in $\mathsf{GL}(\R^3)$. Therefore, a Margulis Space Time comes with a parallel Lorentz metric.

The parallelism classes of timelike geodesics of $\mathsf{M}$ can be parametrized by a non-compact complete hyperbolic surface $\Sigma$. Recent work by Danciger, Gueritaud and Kassel in \cite{dgk} have shown that $\mathsf{M}$ is a $\R$-bundle over $\Sigma$ and the fibers are time like geodesics. 

Previous works of Jones, Charette, Goldman, Labourie and Margulis in \cite{jones}, \cite{labourie invariant} and \cite{geodesic} showed that the dynamics of $\mathsf{M}$ is closely related to that of $\Sigma$. Jones, Charette and Goldman showed in \cite{jones} that bispiralling geodesics in $\mathsf{M}$ exists and they correspond to bispiralling geodesics in $\Sigma$. Goldman and Labourie showed in \cite{geodesic} that non-wandering geodesics in $\mathsf{M}$ correspond to non-wandering geodesics in $\Sigma$.

In this paper, we first chalk out some preliminary notions, in order to prepare the grounds to explicitly describe the stable and unstable laminations of $\mathsf{U}_{\hbox{\tiny rec}}\M$, the space of non-wandering geodesics in $\mathsf{M}$, under the geodesic flow. We carry on to show that the stable lamination contracts under the forward flow and the unstable lamination contracts under the backward flow. More precisely, we prove the following,
\begin{theorem}\label{mainthm1}
Let $\underline{\mathcal{L}}^+$ and $\underline{\mathcal{L}}^-$ be two laminations of the metric space $\mathsf{U}_{\hbox{\tiny rec}}\M$ as defined in definition \ref{lem}. The geodesic flow on the space of non-wandering geodesics in $\mathsf{M}$ contracts $\underline{\mathcal{L}}^+$ exponentially in the forward direction of the flow and contracts $\underline{\mathcal{L}}^-$ exponentially in the backward direction of the flow.
\end{theorem}

Moreover, in the last section using a natural extension of the definition of Anosov representation given in section 2.0.7 of \cite{orilab} we define the notion of an Anosov representation in our context replacing manifolds by metric spaces. Using this definition we can restate our theorem by the following theorem:

\begin{theorem}\label{geomano}
Let $\mathsf{N}$ be the space of all oriented space-like affine lines in the three dimentional affine space and let $\mathcal{L}$ be the orbit foliation of the flow $\Phi$ on $\mathsf{U}_{\hbox{\tiny $\mathrm{rec}$}}\mathsf{M}$. Then there exist a pair of foliations on $\mathsf{N}$ so that $(\mathsf{U}_{\hbox{\tiny $\mathrm{rec}$}}\mathsf{M},\mathcal{L})$ admits a geometric $(\mathsf{N}, \mathsf{SO}^0(2,1)\ltimes\mathbb{R}^3)$ Anosov structure.
\end{theorem}
In other words, mondromy of Margulis Space Times are ``Anosov representations in non semi-simple Lie groups".\\

\textbf{Acknowledgments:} I would like to express my gratitude towards my advisor Prof. Francois Labourie for his guidance. I would like to thank Andres Sambarino for the many helpful discussions that we had. I would also like to thank Thierry Barbot for his careful eye in finding a gap in a previous unpublished version of this work.

\section{Background}

\subsection{Affine Geometry}

An $\textit{affine}$ $\textit{space}$ is a set $\mathbb{A}$ together with a vector space $\mathbb{V}$ and a faithful and transitive group action of $\mathbb{V}$ on $\mathbb{A}$. We call $\mathbb{V}$ the underlying vector space of $\mathbb{A}$ and refer to its elements as translations. An $\textit{affine}$ $\textit{transformation}$ $F$ between two affine spaces $\mathbb{A}_1$ and $\mathbb{A}_2$, is a map such that for all $x$ in $\mathbb{A}_1$ and for all $v$ in $\mathbb{V}_1$, $F$ satisfies the following property: 
\begin{equation} \label{1}
F(x + v) = F(x) + \mathtt{L}(F).v
\end{equation}
for some linear transformation $\mathtt{L}(F)$ between $\mathbb{V}_1$ and $\mathbb{V}_2$. Therefore, by fixing an origin $O$ in $\mathbb{A}$, one can represent an affine transformation $F$, from $\mathbb{A}$ to itself as a combination of a linear transformation and a translation. More precisely, 
\begin{equation} \label{2}
F(O + v) = O + \mathtt{L}(F).v + \left(F(O)-O\right).
\end{equation}
We denote $(F(O)-O)$ by $\mathtt{u}(F)$. Let us denote the space of affine automorphisms of $\mathbb{A}$ onto itself by $\mathsf{Aff}(\mathbb{A})$.\\
Let $\mathsf{GL}(\mathbb{V})$ be the general linear group of $\mathbb{V}$. We consider the semidirect product $\mathsf{GL}(\mathbb{V})\ltimes \mathbb{V}$  of the two groups $\mathsf{GL}(\mathbb{V})$ and $\mathbb{V}$ where the multiplication is defined by
\begin{align}
(g_1, v_1).(g_2, v_2) \defeq (g_1g_2, v_1 + g_1.v_2)
\end{align}
for $g_1, g_2$ in $\mathsf{GL}(\mathbb{V})$ and $v_1, v_2$ in $\mathbb{V}$. Using equation \ref{2} we obtain that the following map:
\[F\mapsto(\mathtt{L}(F),\mathtt{u}(F))\]
defines an isomorphism between $\mathsf{Aff}(\mathbb{A})$ and $\mathsf{GL}(\mathbb{V})\ltimes \mathbb{V}$.\\
Let us denote the tangent bundle of $\mathbb{A}$ by $\mathsf{T}\mathbb{A}$. The tangent bundle $\mathsf{T}\mathbb{A}$ of an affine space $\mathbb{A}$ is a trivial bundle and is canonically isomorphic to $\mathbb{A} \times \mathbb{V}$ as a bundle. 
The geodesic flow $\tilde{\Phi}$ on $\mathsf{T}\mathbb{A}$ is defined as follows,
\begin{align}
\tilde{\Phi}_t \colon \mathsf{T}\mathbb{A} &\longrightarrow \mathsf{T}\mathbb{A} \\
\notag (p,v) &\mapsto (p +tv,v).
\end{align}

\subsection{Hyperboloid Model of Hyperbolic Geometry}

Let $\left(\mathbb{R}^{2,1}, \langle\mid\rangle\right)$ be a Minkowski Space Time where the quadratic form corresponding to the metric $\langle\mid\rangle$ is given by 
\begin{align}\label{lorentz}
\mathcal{Q} \defeq 
\begin{pmatrix}
    1 & 0 & 0\\
    0 & 1 & 0\\
    0 & 0& -1\\
\end{pmatrix}.
\end{align}
Let $\mathsf{SO}(2,1)$ denote the group of linear transformations of $\mathbb{R}^{2,1}$ preserving the metric $\langle\mid\rangle$ and $\mathsf{SO}^{0}(2,1)$ be the connected component containing the identity of $\mathsf{SO}(2,1)$.\\
The cross product $\boxtimes$ associated with this quadratic form is defined as follows:
\begin{align}
u \boxtimes v \defeq (u_2v_3 - u_3v_2 , u_3v_1 - u_1v_3 , u_2v_1 - u_1v_2)^{t}
\end{align} 
where $u,v$ is denoted by $(u_1, u_2, u_3)^{t}$ and  $(v_1, v_2, v_3)^{t}$ respectively. The cross product $\boxtimes$ satisfies the following properties for all $u, v$ in $\mathbb{R}^{2,1}$:
\begin{align}
&\notag\langle u, v \boxtimes w \rangle = \det [u,v,w],\\
&\label{box}\langle u \boxtimes v, u \boxtimes v\rangle = \langle u, v\rangle ^2 - \langle u, u \rangle \langle v, v\rangle ,\\
&\notag u\boxtimes v = - v \boxtimes u.
\end{align}
Now for all real number $k$ we define,
\begin{align*} 
\mathsf{S}^{k} \defeq \lbrace v \in \mathbb{R} \mid \langle v,v \rangle \ = k \rbrace .
\end{align*}
We note that $\mathsf{S}^{-1}$ has two components. We denote the component containing ($0,0,1$)$^{t}$ as $\mathbb{H}$. The quadratic form gives rise to a Riemannian metric of constant negative curvature on the submanifold $\mathbb{H}$ of $\mathbb{R}^{2,1}$. The space $\mathbb{H}$ is called the $\textit{hyperboloid}$ $\textit{model}$ $\textit{of}$ $\textit{hyperbolic}$ $\textit{geometry}$. Let $\mathsf{U}\mathbb{H}$ denote the unit tangent bundle of $\mathbb{H}$. The map
\begin{align} \label{theta}
\Theta : \mathsf{SO}^{0}(2,1) &\longrightarrow \mathsf{U}\mathbb{H}\\
\notag g &\longmapsto \left(g(0,0,1)^t, g(0,1,0)^t\right),
\end{align}
gives an analytic identification between $\mathsf{SO}^{0}(2,1)$ and $\mathsf{U}\mathbb{H}$. Let $\tilde{\phi}_t$ denote the geodesic flow on $\mathsf{U}\mathbb{H} \cong \mathsf{SO}^{0}(2,1)$. We note that $\tilde{\phi_t}(g) = g.a(t)$ where
\begin{align}\label{a} 
a(t) \defeq
\left( \begin{array}{ccc}
1 & 0 & 0 \\
0 & \text{cosh(t)} & \text{sinh(t)} \\
0 & \text{sinh(t)} & \text{cosh(t)} \end{array} 
\right).
\end{align}
We also note that $\tilde{\phi}_t$ is the image of the geodesic flow on $\mathsf{PSL}(2,\mathbb{R})$ under the identification of $\mathsf{PSL}(2,\mathbb{R})$ and $\mathsf{SO}^{0}(2,1)$.

There is a canonical metric $d_{\mathsf{U}\mathbb{H}}$ on the unit tangent bundle $\mathsf{U}\mathbb{H}$ whose restriction on $\mathbb{H}$ is the hyperbolic metric. The metric $d_{\mathsf{U}\mathbb{H}}$ is unique upto the action of the maximal compact subgroup of $\mathsf{SO}^{0}(2,1)$. Let $g\in\mathsf{SO}^{0}(2,1)\cong\mathsf{U}\mathbb{H}$. We recall that the $\textit{horocycles}$ $\tilde{\mathcal{H}}_g^\pm$ for the geodesic flow $\tilde{\phi}$ passing through the point $g$ is defined as follows:
\begin{align}
\tilde{\mathcal{H}}_g^+ &\defeq \{h\in\mathsf{U}\mathbb{H}\mid\lim_{t\to\infty}d_{\mathsf{U}\mathbb{H}}(\tilde{\phi}_t g,\tilde{\phi}_t h)=0 \},\\
\tilde{\mathcal{H}}_g^- &\defeq \{h\in\mathsf{U}\mathbb{H}\mid\lim_{t\to-\infty}d_{\mathsf{U}\mathbb{H}}(\tilde{\phi}_t g,\tilde{\phi}_t h)=0 \}.
\end{align}
We note that under the identification $\Theta$, the horocycle $\tilde{\mathcal{H}}_g^{\pm}$ passing through $g$ is given by $g.u^\pm(t)$, where $u^\pm(t)$ are defined as follows:
\begin{align} \label{u1}
u^+(t) \defeq
\left( \begin{array}{ccc}
1 & -2t & 2t \\
2t & 1-2t^2 & 2t^2 \\
2t & -2t^2 & 1+2t^2 
\end{array} \right),\\
u^-(t) \defeq \label{u2}
\left( \begin{array}{ccc}
1 & 2t & 2t \\
-2t & 1-2t^2 & -2t^2 \\
2t & 2t^2 & 1+2t^2 
\end{array} \right).
\end{align}
We also note that $\tilde{\mathcal{H}}^\pm$ is the image of the horocycles of $\mathsf{PSL}(2,\mathbb{R})$ under the identification of $\mathsf{PSL}(2,\mathbb{R})$ and $\mathsf{SO}^{0}(2,1)$.

Let ${\nu}$ be defined as follows:
\begin{align}
{\nu} \colon  \mathsf{SO}^0(2,1) &\longrightarrow  \mathsf{S}^{1}\\
\notag g &\longmapsto g(1,0,0)^t,
\end{align}
and also let ${\nu}^\pm$ be defined as follows:
\begin{align}
{\nu}{^\pm} \colon  \mathsf{SO}^0(2,1) &\longrightarrow  \mathsf{S}^0\\
\notag g &\longmapsto g.\left(0,\pm\frac{1}{\sqrt{2}},\frac{1}{\sqrt{2}}\right)^t.
\end{align}
The map ${\nu}$ is called the $\textit{neutral section}$ and the maps ${\nu}^+$(respectively ${\nu}^-$) are called the $\textit{positive}$ (respectively $\textit{negative}$) $\textit{limit sections}$. We list a few properties of the neutral section and the limit sections as follows:
\begin{align}
&{\nu}(\tilde{\phi_t}g) = {\nu}(g),\label{nu1}\\
&{\nu}(h.g) = h.{\nu}(g),\label{nu2}\\
&{\nu}^{\pm}(\tilde{\phi_t}g) = e^{\pm t}{\nu}^{\pm}(g),\label{limit1}\\
&{\nu}^{\pm}(h.g) = h.{\nu}^{\pm}(g),\label{limit2}\\
&{\nu}^{+}(g.u^+(t)) = {\nu}^{+}(g),\label{limit3}\\
&{\nu}^{-}(g.u^-(t)) = {\nu}^{-}(g).\label{limit4}
\end{align}
where $t\in\mathbb{R}$ and $g, h\in\mathsf{SO}^0(2,1)$.

Let $\Gamma$ be a free, nonabelian subgroup with finitely many generators. We consider the left action of $\Gamma$ on $\mathsf{U}\mathbb{H}$. We notice that the action of $\Gamma$ being from the left and the action of $a(t)$ being from the right, the two actions commute. Furthermore, given a free and proper action of $\Gamma$ on $\mathsf{U}\mathbb{H}$, one gets an isomorphism between $\Gamma \backslash \mathsf{U}\mathbb{H}$  and $\mathsf{U}\Sigma$, where $\mathsf{U}\Sigma$ is the unit tangent bundle of the surface $\Sigma\defeq\Gamma \backslash \mathbb{H}$. We note that the flow $\tilde{\phi}$ on $\mathsf{U}\mathbb{H}$ gives rise to a flow $\phi$ on $\mathsf{U}\Sigma$.

Let $x_0$ be a point in $\mathbb{H}$. Let $\Gamma .x_0$ denote the orbit of $x_0$ under the action of $\Gamma$. We denote the closure of $\Gamma .x_0$ inside the closure of $\mathbb{H}$ by $\overline{\Gamma .x}_0$. We define the $\textit{limit set}$ of the group $\Gamma$ to be the space $\overline{\Gamma .x}_0\backslash\Gamma .x_0$ and denote it by $\Lambda_{\infty}\Gamma$. We note that the collection $\overline{\Gamma .x}_0\backslash\Gamma .x_0$ is independent of the particular choice of $x_0$. We also know that $\Lambda_{\infty}\Gamma$ is compact.

A point $g\in \mathsf{U}\Sigma$ is called a $\textit{wandering point}$ of the flow $\phi$ if there exists an $\epsilon$-neighborhood $\mathcal{B}_\epsilon(g)\subset\mathsf{U}\Sigma$ around $g$ and a real number $t_0$ such that for all $t>t_0$ we have that
\[\mathcal{B}_\epsilon(g)\cap\phi_t\mathcal{B}_\epsilon(g)=\emptyset .\]
Moreover, a point is called $\textit{non-wandering}$ if it is not a wandering point.

Let $\mathsf{U}_{\hbox{\tiny $\mathrm{rec}$}}\Sigma$ be the space of all non-wandering points of the geodesic flow $\phi$ on $\mathsf{U}\Sigma$. We denote the lift of the space $\mathsf{U}_{\hbox{\tiny $\mathrm{rec}$}}\Sigma$ in $\mathsf{U}\mathbb{H}$ by $\mathsf{U}_{\hbox{\tiny $\mathrm{rec}$}}\mathbb{H}$. Now if the action of $\Gamma$ on $\mathbb{H}$ is free and proper and moreover $\Gamma$ contains no parabolics, then the space $\mathsf{U}_{\hbox{\tiny $\mathrm{rec}$}}\Sigma$ is compact. We note that the subspace $\mathsf{U}_{\hbox{\tiny $\mathrm{rec}$}}\mathbb{H}$ can also be given an alternate description as follows:
\begin{align*}
\mathsf{U}_{\hbox{\tiny $\mathrm{rec}$}}\mathbb{H} = \left\lbrace (x,v) \in \mathsf{U}\mathbb{H} \mid  \lim_{t \to \pm\infty}\tilde{\phi}^1_t x \in \Lambda_{\infty}\Gamma\right\rbrace
\end{align*}
where $\tilde{\phi}_t(x,v) = (\tilde{\phi}^1_tx, \tilde{\phi}^2_tv)$. Furthermore, we note that the space $\mathsf{U}_{\hbox{\tiny $\mathrm{rec}$}}\mathbb{H}$ can be identified with the space $\left(\Lambda_{\infty}\Gamma\times\Lambda_{\infty}\Gamma\setminus\{(x,x)\mid x\in \Lambda_{\infty}\Gamma\} \right)\times\mathbb{R}$.

\subsection{Metric Anosov Property}

The definitions in this section, which can also be found in subsection 3.2 of \cite{pressure metric}, has been included here for the sake of completeness. 

\begin{definition}\label{lam}
Let $(\mathcal{X},d)$ be a metric space. A $\textit{lamination}$ $\mathcal{L}$ of $\mathcal{X}$ is an equivalence relation on $\mathcal{X}$ such that for all $x$ in $\mathcal{X}$ there exist an open neighborhood $\mathcal{U}_x$ of $x$ in $\mathcal{X}$, two topological spaces $\mathcal{U}_1$ and $\mathcal{U}_2$ and a homeomorphism $\mathfrak{f}_x$ from $\mathcal{U}_1 \times \mathcal{U}_2$ onto $\mathcal{U}_x$ satisfying the following properties, 
\begin{enumerate}
\item for all $w, z$ in $\mathcal{U}_x\ \cap\ \mathcal{U}_y$ we have $p_2\left(\mathfrak{f}^{-1}_x(w)\right) = p_2\left(\mathfrak{f}^{-1}_x(z)\right)$ if and only if $p_2\left(\mathfrak{f}^{-1}_y(w)\right) = p_2\left(\mathfrak{f}^{-1}_y(z)\right)$ where $p_2$ is the projection from $\mathcal{U}_1 \times \mathcal{U}_2$ onto $\mathcal{U}_2$,
\item for all $w, z$ in $\mathcal{X}$ we have $w\mathcal{L}z$ if and only if there exists a finite sequence of points $w_1, w_2,..., w_n$ in $\mathcal{X}$ with $w_1 = w$ and $w_n = z$, such that $w_{i+1}$ is in $\mathcal{U}_{w_{i}}$, where $\mathcal{U}_{w_{i}}$ is a neighborhood of $w_i$ and $p_2\left(\mathfrak{f}^{-1}_{w_i}(w_i)\right) = p_2\left(\mathfrak{f}^{-1}_{w_i}(w_{i+1})\right)$ for all $i$ in $\{1,2,...,n-1\}$.
\end{enumerate}
\end{definition}
The homeomorphism $\mathfrak{f}_x$ is called a $\textit{chart}$ and the equivalence classes are called the $\textit{leaves}$.\\
A $\textit{plaque open set}$ in the chart corresponding to $\mathfrak{f}_x$ is a set of the form $\mathfrak{f}_x(\mathcal{V}_1 \times \{x_2\})$ where $x = \mathfrak{f}_x(x_1,x_2)$ and $\mathcal{V}_1$ is an open set in $\mathcal{U}_1$. The $\textit{plaque topology}$ on $\mathcal{L}_x$ is the topology generated by the plaque open sets. A plaque neighborhood of $x$ is a neighborhood for the plaque topology on $\mathcal{L}_x$.
\begin{definition}
A $\textit{local product structure}$ on $\mathcal{X}$ is a pair of two laminations $\mathcal{L}_1$, $\mathcal{L}_2$ satisfying the following property: for all $x$ in $\mathcal{X}$ there exist two plaque neighborhoods $\mathcal{U}_1$, $\mathcal{U}_2$ of $x$,  respectively in $\mathcal{L}_1$, $\mathcal{L}_2$ and a homeomorphism $\mathfrak{f}_x$ from $\mathcal{U}_1 \times \mathcal{U}_2$ onto a neighborhood $\mathcal{W}_x$ of $x$, such that $\mathfrak{f}_x$ defines a chart for both the laminations $\mathcal{L}_1$ and $\mathcal{L}_2$.
\end{definition}
Now, let us assume that $\psi_t$ be a flow on $\mathcal{X}$. A lamination $\mathcal{L}$ invariant under the flow $\psi_t$ is called $\textit{transverse}$ to the flow, if for all $x$ in $\mathcal{X}$, there exists a plaque neighborhood $\mathcal{U}_x$ of $x$ in $\mathcal{L}_x$, a topological space $\mathcal{V}$, a positive $\epsilon$ and a homeomorphism $\mathfrak{f}_x$ from $\mathcal{U}_x \times \mathcal{V} \times (-\epsilon,\epsilon)$ onto an open neighborhood $\mathcal{W}_x$ of $x$ in $\mathcal{X}$ satisfying the following condition:
\[\psi_t(\mathfrak{f}_x(u,v,s)) = \mathfrak{f}_x(u,v,s + t)\]
for $u$ in $\mathcal{U}_x$, $v$ in $\mathcal{V}$ and for $s, t$ in the interval $(-\epsilon,\epsilon)$. Let $\mathcal{L}^{.}$ be a lamination which is transverse to the flow $\psi_t$. We define a new lamination $\mathcal{L}^{.,0}$, called the $\textit{central lamination}$, starting from $\mathcal{L}^{.}$ as follows, we say $y, z$ in $\mathcal{X}$ belongs to the same equivalence class of $\mathcal{L}^{.,0}$ if for some real number $t$, $\psi_ty$ and $z$ belongs to the same equivalence class of $\mathcal{L}^{.}$.
\begin{definition}
A lamination $\mathcal{L}$ invariant under a flow $\psi_t$ is said to $\textit{contract}$ $\textit{under}$ $\textit{the}$ $\textit{flow}$ if there exists a positive real number $T_0$ such that for all $x$ in $\mathcal{X}$, the following holds: there exists a chart $\mathfrak{f}_x$ of an open neighbourhood $W_x$ of $x$, and for any two points $y, z$ in $W_x$ with $y, z$ being in the same equivalence class of $\mathcal{L}$, we have,
\[d(\psi_ty,\psi_tz) < \frac{1}{2}d(y,z)\]
for all $t>T_0$.
\end{definition}
\begin{remark}
We note that a lamination `contracts under a flow' if and only if the lamination contracts exponentially under the flow.  
\end{remark}
\begin{definition}
A flow $\psi_t$ on a compact metric space is called $\textit{Metric Anosov}$, if there exist two laminations $\mathcal{L}^+$ and $\mathcal{L}^-$ of $\mathcal{X}$ such that the following conditions hold:
\begin{align*}
1&.\ (\mathcal{L}^+,\mathcal{L}^{-,0})\text{ defines a local product structure on }\mathcal{X},\\
2&.\ (\mathcal{L}^-,\mathcal{L}^{+,0})\text{ defines a local product structure on }\mathcal{X},\\
3&. \text{ the leaves of }\mathcal{L}^+\text{ are contracted by the flow,}\\
4&. \text{ the leaves of }\mathcal{L}^-\text{ are contracted by the inverse flow.}
\end{align*}
\end{definition}
In such a case we call $\mathcal{L}^+$, $\mathcal{L}^-$, $\mathcal{L}^{+,0}$ and $\mathcal{L}^{-,0}$ respectively the $\textit{stable}$, $\textit{unstable}$, $\textit{central stable}$ and $\textit{central unstable}$ laminations. 

\subsection{Margulis Space Times and Surfaces}
A $\textit{Margulis Space Time}$ $\mathsf{M}$ is a quotient manifold of the three dimensional affine space $\mathbb{A}$ by a free, non-abelian group $\Gamma$ which acts freely and properly as affine transformations with discrete linear part. In \cite{marg1} and \cite{marg2} Margulis showed the existence of these spaces. Later in \cite{D} Drumm introduced the notion of $\textit{crooked planes}$ and constructed fundamental domains of a certain class of Margulis Space Times. In his construction the crooked planes give the boundary of appropriate fundamental domains for a certain class of Margulis Space Times. Recently, in \cite{dgk} Danciger--Gueritaud--Kassel showed that for any Margulis Space Time one can find a fundamental domain whose boundaries are given by union of crooked planes.

If $\Gamma$ is a subgroup of $\mathsf{GL}(\mathbb{R}^3)\ltimes\mathbb{R}^3$ such that $\mathsf{M}\defeq\Gamma\backslash\mathbb{A}$ is a Margulis Space Time then by a result proved by Fried--Goldman in \cite{fried} we get that a conjugate of $\mathtt{L}(\Gamma)$ is a subgroup of $\mathsf{SO}^0(2,1)$. Therefore without loss of generality we can take $\Gamma\subset\mathsf{G}\defeq\mathsf{SO}^0(2,1)\ltimes\mathbb{R}^3$ where $\Gamma$ is a free non-abelian group with finitely many generators. In this thesis I will only consider Margulis Space Times such that $\mathtt{L}(\Gamma)$ contains no parabolic elements.

Let $\mathsf{M}\defeq\Gamma\backslash\mathbb{A}$ be a Margulis Space Time such that $\mathtt{L}(\Gamma)$ contains no parabolic elements. Then the action of $\mathtt{L}(\Gamma)$ on $\mathbb{H}$ is Schottky. Hence $\Sigma\defeq\mathtt{L}(\Gamma)\backslash\mathbb{H}$ is a non-compact surface with no cusps. 

Now let $\mathsf{T}\mathsf{M}$ be the tangent bundle of $\mathsf{M}$. As $\mathtt{L}(\Gamma)\subset\mathsf{SO}^0(2,1)$ we have that $\mathsf{T}\mathsf{M}$ carries a Lorentzian metric $\langle\mid\rangle$. Let
\[\mathsf{U}\mathsf{M}\defeq\{(X,v)\in \mathsf{T}\mathsf{M}\mid\langle v\mid v\rangle_X=1\}.\]
We note that $\mathsf{U}\mathsf{M}\cong\Gamma\backslash\mathsf{U}\mathbb{A}$ where $\mathsf{U}\mathbb{A}\defeq\mathbb{A}\times\mathsf{S}^{1}$. The geodesic flow $\tilde{\Phi}$ on $\mathsf{T}\mathbb{A}$ gives rise to a flow $\Phi$ on $\mathsf{U}\mathsf{M}$. 

We recall that a point $(X,v)\in \mathsf{U}\mathsf{M}$ is called a $\textit{wandering point}$ of the flow $\Phi$ if there exists an $\epsilon$-neighborhood $\mathcal{B}_\epsilon(X,v)\subset\mathsf{U}\Sigma$ around $(X,v)$ and a real number $t_0$ such that for all $t>t_0$ we have that
\[\mathcal{B}_\epsilon(X,v)\cap\Phi_t\mathcal{B}_\epsilon(X,v)=\emptyset .\]
Moreover, a point is called $\textit{non-wandering}$ if it is not a wandering point.

We denote the space of all non-wandering points of the flow $\Phi$ on $\mathsf{U}\mathsf{M}$ by $\mathsf{U}_{\hbox{\tiny $\mathrm{rec}$}}\mathsf{M}$. Moreover, we denote the lift of $\mathsf{U}_{\hbox{\tiny $\mathrm{rec}$}}\mathsf{M}$ into $\mathsf{U}\mathbb{A}$ by $\mathsf{U}_{\hbox{\tiny $\mathrm{rec}$}}\mathbb{A}$. 

In \cite{labourie invariant} Goldman--Labourie--Margulis proved the following theorem:
\begin{theorem}{[Goldman--Labourie--Margulis]} \label{N}
Let $\Gamma$ be a non-abelian free discrete subgroup of $\mathsf{G}$ with finitely many generators giving rise to a Margulis Space Time and let $\mathtt{L}(\Gamma)$ contains no parabolic elements. Then there exists a map 
\[{N}: \mathsf{U}_{\hbox{\tiny $\mathrm{rec}$}}\mathbb{H}\longrightarrow\mathbb{A}\]
and a positive H\"older continuous function
\[{f}: \mathsf{U}_{\hbox{\tiny $\mathrm{rec}$}}\mathbb{H}\longrightarrow\mathbb{R}\]
such that
\begin{enumerate}
\item for all $\gamma\in\Gamma$ we have ${f}\circ\mathtt{L}(\gamma) = {f}$,
\item for all $\gamma\in\Gamma$ we have ${N}\circ\mathtt{L}(\gamma) = \gamma{N}$, and
\item for all $g\in\mathsf{U}_{\hbox{\tiny $\mathrm{rec}$}}\mathbb{H}$ and for all $t\in\mathbb{R}$ we have 
\[{N}(\tilde{\phi_t}g) = {N}(g) + \left(\int\limits_{0}^{t}{f}(\tilde{\phi_s}(g))ds\right){\nu}(g).\]
\end{enumerate}
\end{theorem}
We call ${N}$ a $\textit{neutralised section}$. Using the existence of a neutralised section Goldman--Labourie proved the following theorem in \cite{geodesic}:
\begin{theorem}{[Goldman--Labourie]} \label{commute}
Let $\Gamma$ be a non-abelian free discrete subgroup of $\mathsf{G}$ with finitely many generators giving rise to a Margulis Space Time and let $\mathtt{L}(\Gamma)$ contains no parabolic elements. Also let $\mathsf{U}_{\hbox{\tiny $\mathrm{rec}$}}\Sigma$ and $\mathsf{U}_{\hbox{\tiny $\mathrm{rec}$}}\mathsf{M}$ be defined as above. Now if ${N}$ is a neutralised section, then there exists an injective map $\hat{\mathtt{N}}$ such that the following diagram commutes,
\[
\begin{CD}
\mathsf{U}_{\hbox{\tiny $\mathrm{rec}$}}\mathbb{H} @> \mathtt{N} >> \mathsf{U}\mathbb{A} \\
@ V{\pi} VV @ VV {\pi} V \\
{\mathsf{U}_{\hbox{\tiny $\mathrm{rec}$}}\Sigma} @> {\hat{\mathtt{N}}} >> {\mathsf{U} \mathsf{M}} \\
\end{CD}
\]
where $\mathtt{N}\defeq(N,\nu)$. Moreover, $\hat{\mathtt{N}}$ is an orbit equivalent H\"older homeomorphism onto $\mathsf{U}_{\hbox{\tiny $\mathrm{rec}$}}\mathsf{M}$.
\end{theorem}

\section{Metric Anosov structure on Margulis Space Time}

Let $\mathsf{M}$ be a Margulis Space Time. In this section, first we define a distance function $d$ on $\mathsf{U}_{\hbox{\tiny $\mathrm{rec}$}}\mathsf{M}$ such that $(\mathsf{U}_{\hbox{\tiny $\mathrm{rec}$}}\mathsf{M},d)$ is a metric space. Next, we define two laminations $\mathcal{L}^\pm$ on the metric space $(\mathsf{U}_{\hbox{\tiny $\mathrm{rec}$}}\mathsf{M},d)$ which are invariant under the flow $\Phi_t$ on $\mathsf{U}_{\hbox{\tiny $\mathrm{rec}$}}\mathsf{M}$. Finally, we show that the lamination $\mathcal{L}^+$ is a stable lamination and the lamination $\mathcal{L}^-$ is an unstable lamination for the flow $\Phi_t$ on $(\mathsf{U}_{\hbox{\tiny $\mathrm{rec}$}}\mathsf{M},d)$. We note that the method used in this paper to construct the distance function $d$ and to prove contraction properties of the lamination is inspired by \cite{pressure metric}.

\subsection{Metric space structure}

The restriction of any euclidean metric on $\mathbb{A}\times\mathbb{V}$ to the subspace $\mathsf{U}_{\hbox{\tiny $\mathrm{rec}$}}\mathbb{A}$, defines a distance on $\mathsf{U}_{\hbox{\tiny $\mathrm{rec}$}}\mathbb{A}$. We call this distance the $\textit{euclidean distance}$ on $\mathsf{U}_{\hbox{\tiny $\mathrm{rec}$}}\mathbb{A}$. In this section we will define a distance on the space $\mathsf{U}_{\hbox{\tiny $\mathrm{rec}$}}\mathbb{A}$ such that the distance is locally bilipschitz equivalent to any euclidean distance on $\mathsf{U}_{\hbox{\tiny $\mathrm{rec}$}}\mathbb{A}$ and also is $\Gamma$-invariant, so as to get a distance on the quotient space $\mathsf{U}_{\hbox{\tiny $\mathrm{rec}$}}\mathsf{M}$.

We note that any two euclidean metric on $\mathbb{A}\times\mathbb{V}$ are bilipschitz equivalent with each other and hence any two euclidean distances on $\mathsf{U}_{\hbox{\tiny $\mathrm{rec}$}}\mathbb{A}$ are also bilipschitz  equivalent with each other. Fix an euclidean distance $d$ on $\mathsf{U}_{\hbox{\tiny $\mathrm{rec}$}}\mathbb{A}$. The action of $\Gamma$ on the space $\mathbb{A}\times\mathbb{V}$ gives rise to a collection of distances related to $d$ defined as follows: for any $\gamma$ in $\Gamma$ define,
\begin{align}
d_{\gamma}: \mathsf{U}_{\hbox{\tiny $\mathrm{rec}$}}\mathbb{A}&\times\mathsf{U}_{\hbox{\tiny $\mathrm{rec}$}}\mathbb{A}\longrightarrow \mathbb{R}\\
\notag(x,y)&\longmapsto d\left(\gamma ^{-1} x,\gamma ^{-1} y\right)
\end{align} 
Since each element of $\Gamma$ acts as a bilipschitz automorphism with respect to any euclidean distance, any two distances in the family $\{ d_{\gamma} \} _{\gamma \in \Gamma}$ are bilipschitz equivalent with each other.

Compactness of $\mathsf{U}_{\hbox{\tiny $\mathrm{rec}$}}\Sigma$ implies that $\mathsf{U}_{\hbox{\tiny $\mathrm{rec}$}}\mathsf{M}$ is compact and hence we can choose a pre-compact fundamental domain $D$ of $\mathsf{U}_{\hbox{\tiny $\mathrm{rec}$}}\mathsf{M}$ inside $\mathsf{U}_{\hbox{\tiny $\mathrm{rec}$}}\mathbb{A}$ with an open interior. We can also choose a suitable precompact open set $U$ which contains the closure of $D$. We note that properness of the action of $\Gamma$ on $\mathsf{U}_{\hbox{\tiny $\mathrm{rec}$}}\mathbb{A}$ implies that the cover of $\mathsf{U}_{\hbox{\tiny $\mathrm{rec}$}}\mathbb{A}$ by the open sets $\{ \gamma . U \} _{\gamma \in \Gamma}$, is locally finite.

A path joining two points $x$ and $y$ in $\mathsf{U}_{\hbox{\tiny $\mathrm{rec}$}}\mathbb{A}$ is a pair of tuples,
\begin{align*}
\mathcal{P} = \left((z_0,z_1,...,z_n),(\gamma_1,\gamma_2,...,\gamma_n)\right)
\end{align*} 
where $z_i \in \mathsf{U}_{\hbox{\tiny $\mathrm{rec}$}}\mathbb{A}$ and $\gamma_i \in \Gamma$ such that the following two conditions hold,
\begin{align*}
&\text{1. }x = z_0 \in \gamma_1.U \text{ and } y = z_n \in \gamma_{n}.U,\\
&\text{2. }\text{for all } n > i > 0, z_i \in \gamma_{i}.U \cap \gamma_{i+1}.U.
\end{align*}
\begin{definition}
The length of a path is defined by,
\begin{align*}
l(\mathcal{P}) \defeq \sum\limits_{i=0}^{n-1}d_{\gamma_{i+1}}\left(z_i,z_{i+1}\right)
\end{align*}
\end{definition}
\begin{definition}\label{dist}
We then define, 
\begin{align*}
\tilde{d}(x,y) \defeq \text{inf } \{ l(\mathcal{P}) \mid \mathcal{P} \text{ joins } x \text{ and } y \} 
\end{align*} 
\end{definition}
\begin{lemma}
$\tilde{d}$ is a $\Gamma$-invariant pseudo-metric.
\end{lemma}
\begin{proof}
If $\mathcal{P} = ((z_0,z_1,...,z_n),(\gamma_1,\gamma_2,...,\gamma_n))$ is a path joining $\gamma x$ and $\gamma y$, then the path,
\begin{align*}
 \gamma ^{-1}.\mathcal{P} \defeq \left(\left(\gamma ^{-1}z_0, \gamma ^{-1}z_1,...,\gamma ^{-1}z_n\right),\left(\gamma ^{-1}\gamma_1,\gamma ^{-1}\gamma_2,...,\gamma ^{-1}\gamma_n\right)\right)
\end{align*}
is a path joining $x$ and $y$. Moreover, 
\begin{align}
\notag l(\mathcal{P}) &= \sum \limits _{i=0}^{n-1}d_{\gamma _{i+1}}\left(z_i,z_{i+1}\right) = \sum \limits _{i=0}^{n-1}d\left(\gamma _{i+1} ^{-1}z_i,\gamma _{i+1} ^{-1}z_{i+1}\right)\\
\notag &= \sum \limits _{i=0}^{n-1}d\left(\left(\gamma ^{-1} \gamma _{i+1}\right) ^{-1} \gamma ^{-1}z_i,\left(\gamma ^{-1} \gamma _{i+1}\right) ^{-1}\gamma ^{-1}z_{i+1}\right) \\
\notag &= \sum \limits _{i=0}^{n-1}d_{\gamma ^{-1}\gamma _{i+1}}\left(\gamma ^{-1}z_i,\gamma ^{-1}z_{i+1}\right)\\
\notag &= l\left(\gamma ^{-1}.\mathcal{P}\right).
\end{align}
Hence, using the definition of $\tilde{d}$ we get $\tilde{d}(\gamma x,\gamma y)$ is equal to $\tilde{d}(x,y)$. \\
We also notice that $l(\mathcal{P})$ is a sum of distances. So $l(\mathcal{P})$ is non-negative and hence $\tilde{d}$ is non-negative.
\end{proof}

It remains to show that $\tilde{d}$ is a metric and $\tilde{d}$ is locally bilipschitz equivalent to any euclidean distance. As all euclidean distances are bilipschitz equivalent with each other, it suffices to show that $\tilde{d}$ is locally bilipschitz equivalent with $d$.

\begin{lemma}
$\tilde{d}$ is a metric and $\tilde{d}$ is locally bilipschitz equivalent to $d$.
\end{lemma}

\begin{proof}
Let $z$ be a point in $\mathsf{U}_{\hbox{\tiny $\mathrm{rec}$}}\mathbb{A}$. There exists a neighbourhood $V$ of $z$ in $\mathsf{U}_{\hbox{\tiny $\mathrm{rec}$}}\mathbb{A}$ such that
\[
A \defeq \{ \gamma \mid \gamma.U \cap V \neq \emptyset \} 
\]
is a finite set. We fix $V$ and choose a positive real number $\alpha$ so that
\[ \mathop{\bigcup}_{\gamma \in A} \left\lbrace x \mid d_\gamma(z,x)\leqslant \alpha\right\rbrace \subset V. \]
We have seen that any two distances in the family $\{ d_{\gamma} \} _{\gamma \in \Gamma}$ are bilipschitz equivalent with each other. Hence $A$ being a subset of $\Gamma$, any two distances in $A$ are bilipschitz equivalent with each other. Now finiteness of $A$ implies that we can choose a constant $K$ such that for all $\beta_1,\beta_2$ in $A$ we have that $d_{\beta_1}$ and $d_{\beta_2}$ are $K$-bilipschitz equivalent with each other. We set,
\[ W \defeq \mathop\bigcap_{\gamma \in A} \left\lbrace x \mid d_\gamma (z,x) \leqslant \frac{\alpha}{10K}\right\rbrace. \] We note that $W$ is a subset of $V$ because $K$ is bigger than 1.

By construction, if $x,y$ is in $W$ then for all $\gamma$ in $A$ we have, 
\begin{align}\label{4}
d_\gamma(x,y) \leqslant d_\gamma(x,z) + d_\gamma(z,y) \leqslant \frac{\alpha}{5K}.
\end{align}
Now let $x$ be any point in $ W$, $y$ be any general point and 
\[\mathcal{P} = ((z_0,z_1,...,z_n),(\gamma_1,\gamma_2,...,\gamma_n))\] 
be a path joining $x$ and $y$. 

We notice that $x = z_0$ is in $\gamma_1 U$. On the other hand $x$ is also an element of $W$, which is a subset of $V$. Therefore,
\begin{align*}
\gamma_1 U \cap V \neq \emptyset
\end{align*}
Hence $\gamma_1$ is in $A$. If there exists $k$ such that $\gamma_{k}$ is not in $A$ then we choose $j$ to be the smallest $k$ such that $\gamma_k$ is not in $A$.
\begin{align} \label{5}
l(\mathcal{P}) = \sum \limits _{i=0}^{n-1}d_{\gamma _{i+1}}(z_i,z_{i+1})
\geqslant \sum\limits _{i=0}^{j-1}d_{\gamma _{i+1}}(z_i,z_{i+1}).
\end{align}
Now using the fact that $d_{\gamma_{j-1}}$ is K-bilipschitz  equivalent with $d_{\gamma_i}$ for any $\gamma_i$ in $A$ we get,
\begin{align}
\sum\limits _{i=0}^{j-1}d_{\gamma _{i+1}}(z_i,z_{i+1}) \geqslant \frac{1}{K}\sum\limits _{i=0}^{j-1}d_{\gamma _{j-1}}(z_i,z_{i+1}).
\end{align}
Now from the triangle inequality it follows that
\begin{align}
\frac{1}{K}\sum\limits _{i=0}^{j-1}d_{\gamma _{j-1}}(z_i,z_{i+1}) &\geqslant \frac{1}{K}d_{\gamma _{j-1}}(z_0,z_j)\\
\notag &\geqslant \frac{1}{K}(d_{\gamma_{j-1}}(z,z_j)- {d_{\gamma_{j-1}}}(z,z_0)).
\end{align}
The point $z_0 = x$, belongs to $W$ and $\gamma_{j-1}$ belongs to $A$. Therefore, by the definition of $W$ we get that
\begin{align}\label{6}
{d_{\gamma_{j-1}}}(z,z_0) \leqslant \frac{\alpha}{10K}.
\end{align}
We also know that $\gamma_j$ is not in $A$. Hence $\gamma_j . U$ does not intersect with $V$. The point $z_j$ by definition belongs to $\gamma_j . U$ and so $z_j$ is not in $V$. Therefore by the choice of $\alpha$ it follows that
\begin{align}\label{7}
d_{\gamma_{j-1}}(z,z_j) > \alpha .
\end{align}
Using the inequalities \ref{5} and \ref{6} we get that
\begin{align}
\frac{1}{K}\left( d_{\gamma_{j-1}}(z,z_j)- {d_{\gamma_{j-1}}}(z,z_0)\right) &> \frac{1}{K}\left(\alpha - \frac{\alpha}{10K}\right).
\end{align}
Now as $K$ is bigger than 1 we have,
\begin{align}\label{8}
\frac{1}{K}\left(\alpha - \frac{\alpha}{10K}\right) \geqslant \frac{1}{K}\left(\alpha - \frac{\alpha}{10}\right) > \frac{\alpha}{5K}.
\end{align}
Finally, using the inequalities from \ref{5} to \ref{8} we get that if there exists $k$ such that $\gamma_{k}$ is not in $A$ then,
\begin{align}\label{9}
l(\mathcal{P}) > \frac{\alpha}{5K}.
\end{align}\\
On the other hand, if for all $k$ we have $\gamma_k$ in $A$, then for all $\gamma \in A$ we have, 
\begin{align}  
l(\mathcal{P}) &= \sum \limits _{i=0}^{n-1}d_{\gamma _{i+1}}(z_i,z_{i+1}) \geqslant \frac{1}{K}\sum\limits_{i=0}^{n-1}d_{\gamma}(z_i,z_{i+1}).
\end{align}
And using triangle inequality it follows that
\begin{align}
\frac{1}{K}\sum\limits_{i=0}^{n-1}d_{\gamma}(z_i,z_{i+1}) \geqslant \frac{1}{K}d_\gamma(x,y).
\end{align}
Therefore, in the case when for all $k$, $\gamma_k$ is in $A$, we have for all $\gamma$ in $A$,
\begin{align}\label{10}
l(\mathcal{P}) \geqslant \frac{1}{K}d_\gamma(x,y).
\end{align}

Combining the inequalities \ref{9} and \ref{10} and using the definition of $\tilde{d}$ we have that for any point $x$ in $W$, any general point $y$ and for all $\gamma$ in $A$,
\begin{align} \label{11}
\tilde{d}(x,y) \geqslant \frac{1}{K} \inf\left(\frac{\alpha}{5},d_\gamma(x,y)\right). 
\end{align}
Therefore for any point $y$ distinct from $z$ we have,
\begin{align} 
\tilde{d}(z,y) > 0.
\end{align}
The above is true for any arbitrary choice of $z$ and hence it follows that $\tilde{d}$ is a metric. 

Moreover, if $x, y$ are points in $W$ and $\gamma$ is in $A$ then from the inequality \ref{4} we get,
\begin{align*}
d_\gamma(x,y) \leqslant \frac{\alpha}{5K} \leqslant \frac{\alpha}{5}
\end{align*}
and hence for all $x$, $y$ in $W$ and $\gamma$ in $A$,
\begin{align}\label{12}
\inf \left(\frac{\alpha}{5},d_\gamma(x,y)\right) = d_\gamma(x,y).
\end{align}
Therefore, from the inequalities \ref{11} and \ref{12} it follows that for $x$, $y$ in $W$ and for any $\gamma$ in $A$,
\begin{align}
\tilde{d}(x,y) \geqslant \frac{1}{K}d_\gamma(x,y).
\end{align}
We know that there exists $\gamma_a$ such that the point $z$ is inside the open set $\gamma_a . U$. We note that the above defined $\gamma_a$ is also an element of $A$. Finally, we set $W_a$ to be the intersection of of the set $W$ with the set $\gamma_a . U$. Let $x , y$ be any two points in $W_a$. We choose the path $\mathcal{P}_0 = ((x,y),(\gamma_a,\gamma_a))$ and get that
\begin{align*}
\tilde{d}(x,y) = \text{inf } \{ l(\mathcal{P}) \mid \mathcal{P} \text{ joins } x \text{ and } y \}  \leqslant l(\mathcal{P}_0)=d_{\gamma _a}(x,y).
\end{align*}
Hence, $\tilde{d}$ is bilipschitz equivalent to $d_{\gamma _a}$ on $W_a$ and the distance $d$ is bilipschitz equivalent to $d_{\gamma _a}$. Therefore, $d$ is bilipschitz to $\tilde{d}$ on $W_a$. Since $z$ was arbitrarily chosen it follows that $d$ is locally bilipschitz equivalent to $\tilde{d}$.

\end{proof}

\subsection{The lamination and its lift}

In this subsection, we explicitly describe two laminations of $\mathsf{U}_{\hbox{\tiny $\mathrm{rec}$}}\mathbb{A}$ for the flow $\Phi_t$ on $\mathsf{U}_{\hbox{\tiny $\mathrm{rec}$}}\mathbb{A}$ and show that the laminations are equivariant under the action of the flow and the action of $\Gamma$. We will also define the notion of a leaf lift.

Let $Z$ be a point in $\mathsf{U}_{\hbox{\tiny $\mathrm{rec}$}}\mathbb{A}$. We know from the theorem \ref{commute} that for all $Z\in\mathsf{U}_{\hbox{\tiny $\mathrm{rec}$}}\mathbb{A}$ there exists an unique $g\in\mathsf{U}_{\hbox{\tiny $\mathrm{rec}$}}\mathbb{H}$ such that $Z = \mathtt{N}(g)$. 
\begin{definition}\label{lam1}
The positive and central positive partition of $\mathsf{U}_{\hbox{\tiny $\mathrm{rec}$}}\mathbb{A}$ are respectively given by,
\begin{align*}
\mathcal{L}^{+}_{\mathtt{N}(g)} &\defeq \tilde{\mathcal{L}}^{+}_{\mathtt{N}(g)} \cap \mathsf{U}_{\hbox{\tiny $\mathrm{rec}$}}\mathbb{A}\\
\mathcal{L}^{+,0}_{\mathtt{N}(g)} &\defeq \tilde{\mathcal{L}}^{+,0}_{\mathtt{N}(g)} \cap \mathsf{U}_{\hbox{\tiny $\mathrm{rec}$}}\mathbb{A}
\end{align*}
where
\begin{align*}
\tilde{\mathcal{L}}^{+}_{\mathtt{N}(g)} \defeq \lbrace ({N}(g) &+ s_1 {\nu}^+(g), {\nu}(g) + s_2 {\nu}^+(g))\mid s_1, s_2 \in \mathbb{R} \rbrace, \\
\tilde{\mathcal{L}}^{+,0}_{\mathtt{N}(g)} \defeq \{ ({N}(g) &+ s_1 {\nu}^+(g) + t{\nu}(g) , {\nu}(g) + s_2 {\nu}^+(g))\mid t, s_1, s_2 \in \mathbb{R} \}. 
\end{align*}
\end{definition}
\begin{definition}\label{lam2}
The negative  and central negative partition of $\mathsf{U}_{\hbox{\tiny $\mathrm{rec}$}}\mathbb{A}$ are respectively given by,
\begin{align*}
\mathcal{L}^{-}_{\mathtt{N}(g)} \defeq \tilde{\mathcal{L}}^{-}_{\mathtt{N}(g)} \cap \mathsf{U}_{\hbox{\tiny $\mathrm{rec}$}}\mathbb{A}\\
\mathcal{L}^{-,0}_{\mathtt{N}(g)} \defeq \tilde{\mathcal{L}}^{-,0}_{\mathtt{N}(g)} \cap \mathsf{U}_{\hbox{\tiny $\mathrm{rec}$}}\mathbb{A}
\end{align*}
where
\begin{align*}
\tilde{\mathcal{L}}^{-}_{\mathtt{N}(g)} \defeq \{ ({N}(g) &+ s_1 {\nu}^-(g),{\nu}(g) + s_2 {\nu}^-(g))\mid s_1, s_2 \in \mathbb{R} \},\\
\tilde{\mathcal{L}}^{-,0}_{\mathtt{N}(g)} \defeq \{ ({N}(g) &+ s_1 {\nu}^-(g) + t{\nu}(g), {\nu}(g) + s_2 {\nu}^-(g))\mid t, s_1, s_2 \in \mathbb{R} \}.
\end{align*}
\end{definition}

\begin{lemma}\label{nulimit}
Let $g, h$ be two points in $\mathsf{U}\mathbb{H}$ then 
\begin{align*}
h\text{ is in }\bigcup_{t \in \mathbb{R}}\tilde{\mathcal{H}}_{\tilde{\phi_t}g}^+\text{ if and only if }{\nu}(h) = {\nu}(g) + \frac{\langle {\nu}(h), {\nu}^-(g) \rangle}{\langle {\nu}^+(g), {\nu}^-(g) \rangle}{\nu}^+(g).
\end{align*}
\end{lemma}
\begin{proof}
Let $h$ be a point of $\bigcup_{t \in \mathbb{R}}\tilde{\mathcal{H}}_{\tilde{\phi_t}g}^+$. Hence there exist real numbers $t_1, t_2$ such that $h = {\nu}(ga(t_1)u^+(t_2))$. Therefore, we have
\begin{align*}
{\nu}(h) &= {\nu}(ga(t_1)u^+(t_2)) = ga(t_1)u^+(t_2)\colvec{3}{1}{0}{0} = ga(t_1)\colvec{3}{1}{2t_2}{2t_2}\\
&= ga(t_1)\left(\colvec{3}{1}{0}{0} + \colvec{3}{0}{2t_2}{2t_2}\right) = \nu(g) + 2t_2.ga(t_1)\colvec{3}{0}{1}{1}\\
&= \nu(g) + 2t_2(\cosh t_1 + \sinh t_1).g\colvec{3}{0}{1}{1}\\ 
&= \nu(g) + 2\sqrt{2}\ t_2(\cosh t_1 + \sinh t_1).\nu^+(g).
\end{align*}
Now we notice that
\begin{align*}
\langle\nu(h), \nu^-(g)\rangle &= \langle\nu(g) + 2\sqrt{2}\ t_2(\cosh t_1 + \sinh t_1).\nu^+(g), \nu^-(g)\rangle\\ &= 2\sqrt{2}\ t_2(\cosh t_1 + \sinh t_1).\langle\nu^+(g), \nu^-(g)\rangle .
\end{align*}
Combining the above two calculations we get
\begin{align*}
\nu(h) = \nu(g) + \frac{\langle \nu(h), \nu^-(g) \rangle}{\langle \nu^+(g), \nu^-(g) \rangle}\nu^+(g).
\end{align*}
Now let $g, h$ be two points in $\mathsf{U}\mathbb{H}$ satisfying,
\begin{align*}
\nu(h) = \nu(g) + a_1\nu^+(g)
\end{align*}
for some real number $a_1$. Using the definition of $\nu$ and $\nu^+$ we observe that the above equation is equvalent to the following equation,
\begin{align*}
\left(g.u^+\left(\frac{a_1}{2\sqrt{2}}\right)\right)^{-1}.h\colvec{3}{1}{0}{0} = \left(u^+\left(\frac{a_1}{2\sqrt{2}}\right)\right)^{-1}\colvec{3}{1}{a_1/\sqrt{2}}{a_1/\sqrt{2}} = \colvec{3}{1}{0}{0}.
\end{align*}
We know that the only elements of $\mathsf{SO}^0(2,1)$ fixing the vector $\colvec{3}{1}{0}{0}$ are of the form $a(t)$ for some real number $t$. Hence there exist a real number $t_1$ such that
\begin{align*}
\left(g.u^+\left(\frac{a_1}{2\sqrt{2}}\right)\right)^{-1}.h = a(t_1).
\end{align*}
Therefore,
\begin{align*}
h = g.u^+\left(\frac{a_1}{2\sqrt{2}}\right).a(t_1) = g.a(t_1).u^+\left(\frac{a_1\exp(-t_1)}{2\sqrt{2}}\right)
\end{align*}
and the result follows.
\end{proof}
\begin{corollary}\label{nu+}
Let $g, h$ be two points in $\mathsf{U}\mathbb{H}$ and $h$ is in $\bigcup_{t \in \mathbb{R}}\tilde{\mathcal{H}}_{\tilde{\phi_t}g}^+$ then 
\begin{align*}
\frac{\langle \nu(g), \nu^-(h) \rangle}{\langle \nu^+(h), \nu^-(h) \rangle}\nu^+(h) = -\frac{\langle \nu(h), \nu^-(g) \rangle}{\langle \nu^+(g), \nu^-(g) \rangle}\nu^+(g).
\end{align*}
\end{corollary}
\begin{proof}
We know that if $h$ is in $\bigcup_{t \in \mathbb{R}}\tilde{\mathcal{H}}_{\tilde{\phi_t}g}^+$ then $g$ is in $\bigcup_{t \in \mathbb{R}}\tilde{\mathcal{H}}_{\tilde{\phi_t}h}^+$. Therefore using lemma \ref{nulimit} we get
\begin{align*}
\nu(h) = \nu(g) + \frac{\langle \nu(h), \nu^-(g) \rangle}{\langle \nu^+(g), \nu^-(g) \rangle}\nu^+(g)
\end{align*}
and
\begin{align*}
\nu(g) = \nu(h) + \frac{\langle \nu(g), \nu^-(h) \rangle}{\langle \nu^+(h), \nu^-(h) \rangle}\nu^+(h).
\end{align*}
Hence
\begin{align*}
\frac{\langle \nu(g), \nu^-(h) \rangle}{\langle \nu^+(h), \nu^-(h) \rangle}\nu^+(h) = -\frac{\langle \nu(h), \nu^-(g) \rangle}{\langle \nu^+(g), \nu^-(g) \rangle}\nu^+(g).
\end{align*}
\end{proof}
\begin{definition}
For all $g$ in $\mathsf{U}_{\hbox{\tiny $\mathrm{rec}$}}\mathbb{H}$ we define,
\begin{align*}
\mathcal{H}_g^{\pm} \defeq \tilde{\mathcal{H}}_g^{\pm} \cap \mathsf{U}_{\hbox{\tiny $\mathrm{rec}$}}\mathbb{H}.
\end{align*}
\end{definition}
\begin{proposition} \label{14}
The following equations are true for all $g$ in $\mathsf{U}_{\hbox{\tiny $\mathrm{rec}$}}\mathbb{H}$,
\begin{align*}
1.&\ {\mathcal{L}}^{+,0}_{\mathtt{N}(g)} = \left\lbrace\mathtt{N}(h) \mid h \in \bigcup_{t \in \mathbb{R}}{\mathcal{H}}_{\tilde{\phi_t}g}^{+}\right\rbrace,\\
2.&\ {\mathcal{L}}^{-,0}_{\mathtt{N}(g)} = \left\lbrace\mathtt{N}(h) \mid h \in \bigcup_{t \in \mathbb{R}}{\mathcal{H}}_{\tilde{\phi_t}g}^{-}\right\rbrace.
\end{align*}
\end{proposition}
\begin{proof}
We start with defining a function,
\begin{align}
F : \mathsf{U}_{\hbox{\tiny $\mathrm{rec}$}}\mathbb{H} &\times \mathsf{U}_{\hbox{\tiny $\mathrm{rec}$}}\mathbb{H} \rightarrow \mathbb{R}\\
\notag (g,h) &\mapsto \det[({N}(g)-{N}(h)), \nu(g), \nu(h)] .
\end{align}
Using equation \ref{nu1} and theorem \ref{N} we get that
\begin{align} \label{13}
F(\tilde{\phi _t}g, \tilde{\phi _t}h) = F(g,h)
\end{align}
for all $t\in\mathbb{R}$. Again using equation \ref{nu2} and theorem \ref{N} we get that the neutralised section and the neutral section are equivariant under the action of $\Gamma$. Hence for all $\gamma$ in $\Gamma$ we have,
\begin{align}\label{18}
F(\gamma g,\gamma h) &= \det [({N}(\gamma g)-{N}(\gamma h)), \nu(\gamma g), \nu(\gamma h)] \\
\notag &= \det [\gamma({N}(g)-{N}(h)), \gamma \nu(g), \gamma \nu(h))] \\
\notag &= \det [\gamma] \det [({N}(g)-{N}(h)), \nu(g), \nu(h)] \\
\notag &= \det [({N}(g)-{N}(h)), \nu(g), \nu(h)] \\
\notag &= F(g, h).
\end{align}
Now for a fixed real number $c_0$ we consider the space,
\begin{align*}
\mathfrak{K} \defeq \lbrace (g_1,g_2) \mid d_{\mathsf{U}\mathbb{H}}(g_1,g_2) \leqslant c_0 \rbrace \subset \mathsf{U}_{\hbox{\tiny $\mathrm{rec}$}}\mathbb{H} \times \mathsf{U}_{\hbox{\tiny $\mathrm{rec}$}}\mathbb{H}.
\end{align*}
Compactness of $\mathsf{U}_{\hbox{\tiny $\mathrm{rec}$}}\Sigma$ implies that $\mathfrak{K}_{\Gamma}$, the projection of $\mathfrak{K}$ in $\Gamma \backslash (\mathsf{U}_{\hbox{\tiny $\mathrm{rec}$}}\mathbb{H} \times \mathsf{U}_{\hbox{\tiny $\mathrm{rec}$}}\mathbb{H})$, is compact. Now continuity of $F$ implies that $F$ is uniformly continuous on $\mathfrak{K}_{\Gamma}$.

Let $g$ and $h$ be two points in $\mathsf{U}_{\hbox{\tiny $\mathrm{rec}$}}\mathbb{H}$ such that $h$ is in $\mathcal{H}_g^+$. Given any such choice of $g$ and $h$ we can choose a sufficiently large $t_0$ such that $d_{\mathsf{U}\mathbb{H}}(\tilde{\phi}_{t_0} g,\tilde{\phi}_{t_0} h)$ is arbitrarily close to zero, hence we have $F(\tilde{\phi}_{t_0} g,\tilde{\phi}_{t_0} h)$ arbitrarily close to zero. Therefore by using equation \ref{13} it follows that $F(g,h)$ is zero for all $h$ in $\mathcal{H}_g^+$. 

Now using equation \ref{13}, equation \ref{nu1} and lemma \ref{nulimit} we have,
\begin{align*}
0 &= F(\tilde{\phi _t}g, \tilde{\phi _t}h) = \det[({N}(\tilde{\phi _t}g)-{N}(\tilde{\phi _t}h)), \nu(\tilde{\phi _t}g), \nu(\tilde{\phi _t}h)]\\ 
&= \det[({N}(\tilde{\phi _t}g)-{N}(\tilde{\phi _t}h)), \nu(g), \nu(h)]\\
&= \det[({N}(\tilde{\phi _t}g)-{N}(\tilde{\phi _t}h)), \nu(g), \nu(g) + \frac{\langle \nu(h), \nu^-(g) \rangle}{\langle \nu^+(g), \nu^-(g) \rangle}\nu^+(g)]\\
&= \frac{\langle \nu(h), \nu^-(g) \rangle}{\langle \nu^+(g), \nu^-(g) \rangle}\det[({N}(\tilde{\phi _t}g)-{N}(\tilde{\phi _t}h)), \nu(g), \nu^+(g)].
\end{align*}
Therefore for all $h$ in $\mathcal{H}_g^+$ and for all real number $t$ we have 
\begin{align*}
\det[({N}(\tilde{\phi _t}g)-{N}(\tilde{\phi _t}h)), \nu(g), \nu^+(g)] = 0.
\end{align*}
Hence there exist real numbers $a_1, b_1$ such that
\begin{align}\label{19}
{N}(\tilde{\phi _t}h) &= {N}(\tilde{\phi _t}g) + a_1 \nu(g) + b_1 \nu^+(g)\\
\notag&= {N}(g) + \left(a_1 + \int\limits_{0}^{t}{f}(\tilde{\phi_s}(g))ds\right) \nu(g) + b_1 \nu^+(g).
\end{align}
Combining lemma \ref{nulimit} and equation \ref{19} we get that
\begin{align*}
{\mathcal{L}}^{+,0}_{\mathtt{N}(g)} \supseteq \left\lbrace\mathtt{N}(h) \mid h \in \bigcup_{t \in \mathbb{R}}{\mathcal{H}}_{\tilde{\phi_t}g}^+\right\rbrace
\end{align*}
Now let $W\in{\mathcal{L}}^{+,0}_{\mathtt{N}(g)}$. By theorem \ref{commute} we know that there exist $h\in\mathsf{U}_{\hbox{\tiny $\mathrm{rec}$}}\mathbb{H}$ such that $W = \mathtt{N}(h)$. Now the choice of $W$ implies that there exist some real number $a_2$ such that
\begin{align*}
\nu(h) = \nu(g) + a_2 \nu^+(g).
\end{align*}
Using lemma \ref{nulimit} we get that $h\in\bigcup\limits_{t \in \mathbb{R}}\tilde{\mathcal{H}}_{\tilde{\phi_t}g}^+$. Therefore $h$ is in 
\begin{align*}
\bigcup\limits_{t \in \mathbb{R}}{\mathcal{H}}_{\tilde{\phi_t}g}^+ = \left(\mathsf{U}_{\hbox{\tiny $\mathrm{rec}$}}\mathbb{H} \cap \bigcup\limits_{t \in \mathbb{R}}\tilde{\mathcal{H}}_{\tilde{\phi_t}g}^+\right)
\end{align*} and we have
\begin{align*}
{\mathcal{L}}^{+,0}_{\mathtt{N}(g)} \subseteq \left\lbrace\mathtt{N}(h) \mid h \in \bigcup_{t \in \mathbb{R}}{\mathcal{H}}_{\tilde{\phi_t}g}^+\right\rbrace .
\end{align*}
Similarly the other equality follows.
\end{proof}

\begin{proposition} \label{lps}
Let $\mathcal{U}_{\mathtt{N}(g)}\subset\mathsf{U}_{\hbox{\tiny $\mathrm{rec}$}}\mathbb{A}$ be a neighborhood of a point $\mathtt{N}(g)$ in $\mathsf{U}_{\hbox{\tiny $\mathrm{rec}$}}\mathbb{A}$. Then the following map is a local homeomorphism:
\begin{align*}
{\amalg}_{\mathtt{N}(g)} : \mathcal{U}_{\mathtt{N}(g)} &\rightarrow (\Lambda_{\infty}\Gamma \times \Lambda_{\infty}\Gamma \setminus \Delta) \times \mathbb{R}\\
\notag \mathtt{N}(h) &\mapsto \left( h^-, h^+, \left\langle {N}(h)-{N}(g), \nu\left(g^-,h^+\right)\right\rangle \right)
\end{align*}
where $h^\pm\defeq\lim_{t\to\pm\infty}\pi(\tilde{\phi}_t h)$ and $\pi$ is the projection from $\mathsf{U}\mathbb{H}$ onto $\mathbb{H}$.
\end{proposition}
\begin{proof}
Let $g$ be a point in $\mathsf{U}_{\hbox{\tiny $\mathrm{rec}$}}\mathbb{H}$. We note that for $g\in\mathsf{U}_{\hbox{\tiny $\mathrm{rec}$}}\mathbb{H}$ the points $g^{\pm}$ lies in $\Lambda_{\infty}\Gamma$. We observe that $\partial\mathbb{H}\setminus\{g^+\}$ is homeomorpic to $\mathbb{R}$. Given any $g$, let $\mathcal{V}_{g^-}$ denote a connected bounded open neighborhood of $g^-$ in $\partial\mathbb{H}\setminus\{g^+\}$ and $\mathcal{V}_{g^+}$ be a connected open neighborhood of $g^+$ in $\partial\mathbb{H}\setminus\{g^-\}$ such that $\mathcal{V}_{g^-} \cap \mathcal{V}_{g^+}$ is empty and $\mathcal{V}_{g^-} \times \mathcal{V}_{g^+}$ is a subset of $\partial\mathbb{H} \times \partial\mathbb{H} \setminus \Delta$. We define $\mathcal{U}_{g^\pm} \defeq \mathcal{V}_{g^\pm} \cap \Lambda_\infty\Gamma$. Let $\mathcal{U}_{g}$ be the open subset of $\mathsf{U}_{\hbox{\tiny $\mathrm{rec}$}}\mathbb{H}$ corresponding to the open set $\mathcal{U}_{g^-} \times \mathcal{U}_{g^+} \times \mathbb{R}$. We consider the following continuous map,
\begin{align*}
{\mathfrak{N}}_g : \mathcal{U}_{g}  &\longrightarrow \mathbb{A}\\
h &\longmapsto {N}(h) - \left\langle {N}(h)-{N}(g), \nu(g^-,h^+)\right\rangle \nu(h)
\end{align*}
We notice that 
\[\nu(g^-,h^+)=\frac{\nu^-(g)\boxtimes \nu^+(h)}{\langle \nu^-(g),\nu^+(h)\rangle}.\]
Hence for all real number $t$ we have
\[{\mathfrak{N}}_g(\tilde{\phi}_th) = {\mathfrak{N}}_g(h).\] Now we define the following continuous map:
\begin{align*}
{\Pi}_g : \mathcal{U}_{g^-} \times \mathcal{U}_{g^+} \times \mathbb{R} &\longrightarrow \mathsf{U}_{\hbox{\tiny $\mathrm{rec}$}}\mathbb{A}\\
\notag (h^-, h^+, t) &\longmapsto \left({\mathfrak{N}}_g + t\nu, \nu\right)(h^-, h^+, t)
\end{align*}
and conclude by observing that 
\[{\amalg}_{\mathtt{N}(g)}\circ{\Pi}_g = \textsf{Id},\]
\[{\Pi}_g\circ{\amalg}_{\mathtt{N}(g)} = \textsf{Id}.\]
\end{proof}

\begin{proposition}\label{lps1}
Let $\mathcal{L}^+$ be as defined in definition \ref{lam1}. Then $\mathcal{L}^+$ is a lamination of $\mathsf{U}_{\hbox{\tiny $\mathrm{rec}$}}\mathbb{A}$.
\end{proposition}
\begin{proof}
We now show that the equivalence relation $\mathcal{L}^+$ on $\mathsf{U}_{\hbox{\tiny $\mathrm{rec}$}}\mathbb{A}$ satisfy properties (1) and (2) of definition \ref{lam} for the local homeomorphism $\amalg$.

\textit{Property} (1):
Let $g_1, g_2$ be two points in $\mathsf{U}_{\hbox{\tiny $\mathrm{rec}$}}\mathbb{H}$, $h_1, h_2$ be two points in the intersection $\mathcal{U}_{g_1} \cap$ $\mathcal{U}_{g_2}$ and $p^{+,0}$ be the projection from $\mathcal{U}_{g^-}\times\mathcal{U}_{g^+}\times\mathbb{R}$ onto $\mathcal{U}_{g^+} \times \mathbb{R}$.  We notice that if
\begin{align*}
p^{+,0}\circ{\amalg}_{\mathtt{N}(g_1)}(\mathtt{N}(h_1)) = p^{+,0}\circ{\amalg}_{\mathtt{N}(g_1)}(\mathtt{N}(h_2))
\end{align*} 
then $h_1^+ = h_2^+$ and 
\begin{align*}
\left\langle {N}(h_1)-{N}(g_1), \nu(g_1^-,h_1^+)\right\rangle = \left\langle {N}(h_2)-{N}(g_1), \nu(g_1^-,h_2^+) \right\rangle .
\end{align*}
Now using proposition \ref{14}, corollary \ref{nu+} and the fact that $h_1^+ = h_2^+$ we get
\[\nu^+(h_2) = c\nu^+(h_1)\]
\[{N}(h_2) = {N}(h_1) + s\nu^+(h_1) + t\nu(h_1)\]
where $c, s, t \in\mathbb{R}$. Hence for $i\in\{1,2\}$
\[\nu(g_i^-,h_2^+)=\nu(g_i^-,h_1^+).\]
Finally using the fact that
\[\left\langle {N}(h_2)-{N}(h_1), \nu(g_1^-,h_1^+) \right\rangle = 0\]
and
\[\nu(g^-,h^+)=\frac{\nu^-(g)\boxtimes \nu^+(h)}{\langle \nu^-(g),\nu^+(h)\rangle}\]
we get $t = 0$. Therefore 
\[\left\langle {N}(h_2)-{N}(h_1), \nu(g_2^-,h_1^+)\right\rangle = \left\langle s\nu^+(h_1), \frac{\nu^-(g_2)\boxtimes \nu^+(h_1)}{\langle \nu^-(g_2),\nu^+(h_1)\rangle}\right\rangle = 0.\]
Hence
\begin{align*}
\left\langle {N}(h_1)-{N}(g_2), \nu(g_2^-,h_1^+)\right\rangle = \left\langle {N}(h_2)-{N}(g_2), \nu(g_2^-,h_2^+) \right\rangle
\end{align*}
and it follows that
\begin{align*}
p^{+,0}\circ{\amalg}_{\mathtt{N}(g_2)}(\mathtt{N}(h_1)) = p^{+,0}\circ{\amalg}_{\mathtt{N}(g_2)}(\mathtt{N}(h_2)).
\end{align*} 
Similarly if we have
\begin{align*}
p^{+,0}\circ{\amalg}_{\mathtt{N}(g_2)}(\mathtt{N}(h_1)) = p^{+,0}\circ{\amalg}_{\mathtt{N}(g_2)}(\mathtt{N}(h_2))
\end{align*}
then
\begin{align*}
p^{+,0}\circ{\amalg}_{\mathtt{N}(g_1)}(\mathtt{N}(h_1)) = p^{+,0}\circ{\amalg}_{\mathtt{N}(g_1)}(\mathtt{N}(h_2)).
\end{align*}

\textit{Property} (2):
Let $\{\mathtt{N}(h_i)\}_{i\in\{1, 2,..., n\}}$ be a sequence of points such that for all $i\in\{1,2,...,n-1\}$ we have 
\[\mathtt{N}(h_{i+1})\in\mathcal{U}_{\mathtt{N}(h_i)}\] 
and
\begin{align*}
p^{+,0}\circ{\amalg}_{\mathtt{N}(h_i)}(\mathtt{N}(h_i)) = p^{+,0}\circ{\amalg}_{\mathtt{N}(h_i)}(\mathtt{N}(h_{i+1})).
\end{align*} 
Hence we have $h_i^+ = h_{i+1}^+$ and 
\[0 = \left\langle {N}(h_i)-{N}(h_i), \nu(h_i^-,h_i^+)\right\rangle = \left\langle {N}(h_{i+1})-{N}(h_i), \nu(h_i^-,h_{i+1}^+) \right\rangle .\]
Now using proposition \ref{14}, corollary \ref{nu+} and $h_i^+ = h_{i+1}^+$ we get that
\[\nu^+(h_{i+1}) = c_i\nu^+(h_i),\]
\[{N}(h_{i+1}) = {N}(h_i) + s_i\nu^+(h_i) + t_i\nu(h_i)\]
for some real numbers $c_i, s_i$ and $t_i$. Hence
\[\nu(h_i^-,h_i^+)=\nu(h_i^-,h_{i+1}^+).\]
Now using the fact that
\begin{align*}
\left\langle {N}(h_{i+1})-{N}(h_i), \nu(h_i^-,h_{i+1}^+)\right\rangle = 0
\end{align*}
we get $t = 0$. Hence we have 
\begin{align*}
{\mathcal{L}}^{+}_{\mathtt{N}(h_i)} = {\mathcal{L}}^{+}_{\mathtt{N}(h_{i+1})}.
\end{align*}
Therefore we conclude that
\begin{align*}
{\mathcal{L}}^{+}_{\mathtt{N}(h_1)} = {\mathcal{L}}^{+}_{\mathtt{N}(h_n)}.
\end{align*}
Now we show the other direction. Let $h\in\mathsf{U}_{\hbox{\tiny $\mathrm{rec}$}}\mathbb{H}$ such that $\mathtt{N}(h)\in{\mathcal{L}}^{+}_{\mathtt{N}(g)}$. Using proposition \ref{14} we get that $h^+ = g^+$. Let $\mathcal{V}_{g^-}$ be a connected bounded open neighborhood of $g^-$ in $\partial_\infty\mathbb{H}\setminus\{g^+\}$ containing the point $h^-$ and let $\mathcal{V}_{g^+}$ be a connected open neighborhood of $g^+$ in $\partial_\infty\mathbb{H}\setminus\{g^-\}$ such that the intersection $\mathcal{V}_{g^+}\cap\mathcal{V}_{g^-}$ is empty. We denote the sets $\mathcal{V}_{g^\pm} \cap \Lambda_\infty\Gamma$ respectively by $\mathcal{U}_{g^\pm}$, the open subset of $\mathsf{U}_{\hbox{\tiny $\mathrm{rec}$}}\mathbb{H}$ corresponding to the open set $\mathcal{U}_{g^-} \times \mathcal{U}_{g^+} \times \mathbb{R}$ by $\mathcal{U}_{g}$ and the open set $\mathtt{N}(\mathcal{U}_{g})$ around $\mathtt{N}(g)$ by $\mathcal{U}_{\mathtt{N}(g)}$. Now we consider the chart $\left(\mathcal{U}_{\mathtt{N}(g)}, {\amalg}_{\mathtt{N}(g)}\right)$ and notice that 
\[p^{+,0}\circ{\amalg}_{\mathtt{N}(g)}(\mathtt{N}(g)) = \left( g^+, 0 \right).\]
Since $\mathtt{N}(h)\in\mathcal{L}^+_{\mathtt{N}(g)}$, using the definition of $\mathcal{L}^+_{\mathtt{N}(g)}$ we get 
\begin{align*}
\left\langle {N}(h)-{N}(g), \nu(g^-,g^+)\right\rangle = 0.
\end{align*}
Now using corollary \ref{nu+} and the fact that $h^+ = g^+$ we obtain
\[\nu(g^-,g^+)=\nu(g^-,h^+).\]
Hence 
\begin{align*}
\left\langle {N}(h)-{N}(g), \nu(g^-,h^+)\right\rangle = 0
\end{align*}
and we finally have
\[p^{+,0}\circ{\amalg}_{\mathtt{N}(g)}(\mathtt{N}(g)) = p^{+,0}\circ{\amalg}_{\mathtt{N}(g)}(\mathtt{N}(h)).\]
Therefore we conclude that ${\mathcal{L}}^{+}$ defines a lamination with plaque neighborhoods given by the image of the open sets $\mathcal{U}_{g^-}$ for $g^-$ in $\Lambda_\infty\Gamma\setminus\{g^+\}$.
\end{proof}

\begin{proposition}\label{lps2}
Let $\mathcal{L}^{-,0}$ be as defined in definition \ref{lam2}. Then $\mathcal{L}^{-,0}$ is a lamination of $\mathsf{U}_{\hbox{\tiny $\mathrm{rec}$}}\mathbb{A}$. Moreover, it is the central lamination corresponding to the lamination $\mathcal{L}^-$.
\end{proposition}
\begin{proof}
We show that the equivalence relation $\mathcal{L}^{-,0}$ on $\mathsf{U}_{\hbox{\tiny $\mathrm{rec}$}}\mathbb{A}$ satisfy properties (1) and (2) of definition \ref{lam} for the local homeomorphism $\amalg$.

\textit{Property} (1): Let $g_1, g_2$ be two points in $\mathsf{U}_{\hbox{\tiny $\mathrm{rec}$}}\mathbb{H}$, $h_1, h_2$ be two points in the intersection $\mathcal{U}_{g_1} \cap$ $\mathcal{U}_{g_2}$ and $p^{+,0}$ be the projection from $\mathcal{U}_{g^-}\times\mathcal{U}_{g^+}\times\mathbb{R}$ onto $\mathcal{U}_{g^+}\times\mathbb{R}$. We see that
\[p^{-}\circ{\amalg}_{\mathtt{N}(g_1)}(\mathtt{N}(h_1)) = p^{-}\circ{\amalg}_{\mathtt{N}(g_1)}(\mathtt{N}(h_2))\]
if and only if 
\[p^{-}\circ{\amalg}_{\mathtt{N}(g_2)}(\mathtt{N}(h_1)) = p^{-}\circ{\amalg}_{\mathtt{N}(g_2)}(\mathtt{N}(h_2))\]
since we have
\[p^{-}\circ{\amalg}_{\mathtt{N}(g_1)}(\mathtt{N}(h_1)) = h_1^- = p^{-}\circ{\amalg}_{\mathtt{N}(g_2)}(\mathtt{N}(h_1))\]
and
\[p^{-}\circ{\amalg}_{\mathtt{N}(g_1)}(\mathtt{N}(h_2)) = h_2^- = p^{-}\circ{\amalg}_{\mathtt{N}(g_2)}(\mathtt{N}(h_2)).\]

\textit{Property} (2):
Let $\{\mathtt{N}(h_i)\}_{i\in\{1, 2,..., n\}}$ be a sequence of points such that for all $i\in\{1,2,...,n-1\}$ we have 
\[\mathtt{N}(h_{i+1})\in\mathcal{U}_{\mathtt{N}(h_i)}\] 
and 
\[p^{-}\circ{\amalg}_{\mathtt{N}(h_i)}(\mathtt{N}(h_i)) = p^{-}\circ{\amalg}_{\mathtt{N}(h_i)}(\mathtt{N}(h_{i+1})).\]
Hence for all $i\in\{1,2,...,n-1\}$ we have $h_i^- = h_{i+1}^-$. Now using proposition \ref{14} we get that
\[{\mathcal{L}}^{-,0}_{\mathtt{N}(h_i)} = {\mathcal{L}}^{-,0}_{\mathtt{N}(h_{i+1})}\]
for all $i$ in $\{1,2,...,n-1\}$. Hence
\[{\mathcal{L}}^{-,0}_{\mathtt{N}(h_{1})} = {\mathcal{L}}^{-,0}_{\mathtt{N}(h_{n})}.\]
Now we show the other direction. Let $h\in\mathsf{U}_{\hbox{\tiny $\mathrm{rec}$}}\mathbb{H}$ such that $\mathtt{N}(h)\in{\mathcal{L}}^{-,0}_{\mathtt{N}(g)}$. Using proposition \ref{14} we get that $h^- = g^-$. Let $\mathcal{V}_{g^+}$ be a connected bounded open neighborhood of $g^+$ in $\partial_\infty\mathbb{H}\setminus\{g^-\}$ containing the point $h^+$ and let $\mathcal{V}_{g^-}$ be a connected open neighborhood of $g^-$ in $\partial_\infty\mathbb{H}\setminus\{g^+\}$ such that $\mathcal{V}_{g^+}\cap\mathcal{V}_{g^-}$ is empty. We denote the sets $\mathcal{V}_{g^\pm} \cap \Lambda_\infty\Gamma$ respectively by $\mathcal{U}_{g^\pm}$, the open subset of $\mathsf{U}_{\hbox{\tiny $\mathrm{rec}$}}\mathbb{H}$ corresponding to the open set $\mathcal{U}_{g^-} \times \mathcal{U}_{g^+} \times \mathbb{R}$ by $\mathcal{U}_{g}$ and the open set $\mathtt{N}(\mathcal{U}_{g})$ around $\mathtt{N}(g)$ by $\mathcal{U}_{\mathtt{N}(g)}$. Now we consider the chart $\left(\mathcal{U}_{\mathtt{N}(g)}, {\amalg}_{\mathtt{N}(g)}\right)$ and notice that 
\[p^{-}\circ{\amalg}_{\mathtt{N}(g)}(\mathtt{N}(g)) = g^- = h^- = p^{-}\circ{\amalg}_{\mathtt{N}(g)}(\mathtt{N}(h)).\]
Therefore we conclude that ${\mathcal{L}}^{-,0}$ defines a lamination with plaque neighborhoods given by the image of the open sets $\mathcal{U}_{g^+}\times\mathbb{R}$ for $g^+$ in $\Lambda_\infty\Gamma\setminus\{g^+\}$.

Now the fact that ${\mathcal{L}}^{-,0}$ is the central lamination corresponding to the lamination $\mathcal{L}^-$ follows from definition \ref{lam2}.
\end{proof}

\begin{theorem}
The laminations $({\mathcal{L}}^{+},{\mathcal{L}}^{-,0})$ and $({\mathcal{L}}^{-},{\mathcal{L}}^{+,0})$ define a local product structure on $\mathsf{U}_{\hbox{\tiny $\mathrm{rec}$}}\mathbb{A}$.
\end{theorem}
\begin{proof}
Using proposition \ref{lps}, \ref{lps1} and \ref{lps2} we get that $({\mathcal{L}}^{+},{\mathcal{L}}^{-,0})$ defines a local product structure. In a similar way one can show that $({\mathcal{L}}^{-},{\mathcal{L}}^{+,0})$ also defines a local product structure.
\end{proof}

\begin{proposition}
The laminations are equivariant under the action of $\Gamma$.
\end{proposition}
\begin{proof}
Let $Z$ be in $\mathsf{U}_{\hbox{\tiny $\mathrm{rec}$}}\mathbb{A}$ such that $Z = \mathtt{N}(g)$ for some $g\in\mathsf{U}_{\hbox{\tiny $\mathrm{rec}$}}\mathbb{H}$ and $W\in{\mathcal{L}}_Z^{+}$. Therefore there exist real numbers $s_1$, $s_2$ such that \[W = (\tilde{{N}}(g) + s_1 \nu^+(g), \nu(g) + s_2 \nu^+(g)).\] Now for all $\gamma$ in $\Gamma$ we get,
\begin{align*}
\gamma . Z & = \gamma . \mathtt{N}(g)\\
&= \mathtt{N}(\gamma . g)
\end{align*}
and 
\begin{align*}
\gamma . W &= \gamma . (\tilde{{N}}(g) + s_1 \nu^+(g), \nu(g) + s_2 \nu^+(g))\\
&= (\gamma . \tilde{{N}}(g) + s_1 . \gamma . \nu^+(g), \gamma . \nu(g) + s_2 . \gamma . \nu^+(g))\\
&= (\tilde{{N}}(\gamma . g) + s_1 \nu^+(\gamma . g), \nu(\gamma . g) + s_2 \nu^+(\gamma . g)).
\end{align*}
Therefore $\gamma . W\in\tilde{\mathcal{L}}^+_{\gamma . Z}$ and $\mathsf{U}_{\hbox{\tiny $\mathrm{rec}$}}\mathbb{A}$ is invariant under the action of $\Gamma$ implies that $\gamma . W\in{\mathcal{L}}^+_{\gamma . Z}$. Hence we get that for all $\gamma$ in $\Gamma$,
\[{\mathcal{L}}^+_{\gamma . Z} = \gamma . {\mathcal{L}}^+_Z.\]
Similarly one can show that for all $\gamma$ in $\Gamma$, 
\[{\mathcal{L}}^-_{\gamma . Z} = \gamma . {\mathcal{L}}^-_Z.\]
\end{proof}

\begin{proposition}
The laminations are equivariant under the geodesic flow.
\end{proposition}
\begin{proof}
Let $Z$ be in $\mathsf{U}_{\hbox{\tiny $\mathrm{rec}$}}\mathbb{A}$ such that $Z = \mathtt{N}(g)$ for some $g\in\mathsf{U}_{\hbox{\tiny $\mathrm{rec}$}}\mathbb{H}$ and $W\in{\mathcal{L}}_Z^{+}$. Therefore there exist real numbers $s_1$, $s_2$ such that \[W = (\tilde{{N}}(g) + s_1 \nu^+(g), \nu(g) + s_2 \nu^+(g)).\] 
We have for all real number $t$,
\begin{align*}
\tilde{\Phi}_t Z &= \tilde{\Phi}_t \mathtt{N}(g)\\
&= ({N}(g) + t\nu(g), \nu(g))
\end{align*}
and
\begin{align*}
\tilde{\Phi}_t W &= \tilde{\Phi}_t ({N}(g) + s_1 \nu^+(g), \nu(g) + s_2 \nu^+(g))\\
&= ({N}(g) + s_1 \nu^+(g) + t . (\nu(g) + s_2 \nu^+(g)), \nu(g) + s_2 \nu^+(g))\\
&= (({N}(g) + t\nu(g)) + (s_1 + ts_2) \nu^+(g), \nu(g) + s_2 \nu^+(g)).
\end{align*}
Therefore for all real number $t$ we have $\tilde{\Phi}_t . W\in\tilde{\mathcal{L}}^+_{\tilde{\Phi}_t . Z}$ and $\mathsf{U}_{\hbox{\tiny $\mathrm{rec}$}}\mathbb{A}$ is invariant under the geodesic flow implies that $\tilde{\Phi}_t . W\in{\mathcal{L}}^+_{\tilde{\Phi}_t . Z}$. Hence we get that for all real number $t$,
\begin{align*}
{\mathcal{L}}^+_{\tilde{\Phi}_t . Z} = \tilde{\Phi}_t . {\mathcal{L}}^+_Z.
\end{align*}
Similarly one can show that for all real number $t$, 
\begin{align*}
{\mathcal{L}}^-_{\tilde{\Phi}_t . Z} = \tilde{\Phi}_t . {\mathcal{L}}^-_Z.
\end{align*}
\end{proof}
\begin{definition}\label{lem}
We denote the projection of $\mathcal{L}^{\pm},\mathcal{L}^{\pm,0}$ on the space $\mathsf{U}_{\hbox{\tiny $\mathrm{rec}$}}\mathsf{M}$ by $\underline{\mathcal{L}}^{\pm}, \underline{\mathcal{L}}^{\pm,0}$ respectively.
\end{definition}

Now we define the notion of a leaf lift. The leaf lift is a map from the leaves of the lamination through a point, to the tangent space of $\mathsf{U}\mathbb{A}$ at that point. We will use this leaf lift to compare distance between the metric $\tilde{d}$ and the norm on the tangent space on any point of the leaves. We define the leaf lift as follows:\\
The $\textit{positive}$ $\textit{leaf}$ $\textit{lift}$ is the map,
\begin{align}
i^+_{\mathtt{N}(g)} &: \tilde{\mathcal{L}}^{+}_{\mathtt{N}(g)} \longrightarrow \mathsf{T}_{\mathtt{N}(g)}\mathsf{U}\mathbb{A}\\
\notag({N}(g) + s_1 \nu^+(g) &,\nu(g) + s_2 \nu^+(g))
\longmapsto (s_1 \nu^+(g), s_2 \nu^+(g)).
\end{align}
and the $\textit{negative}$ $\textit{leaf}$ $\textit{lift}$ is the map,
\begin{align}
i^-_{\mathtt{N}(g)} &: \tilde{\mathcal{L}}^{-}_{\mathtt{N}(g)} \longrightarrow \mathsf{T}_{\mathtt{N}(g)}\mathsf{U}\mathbb{A}\\
\notag({N}(g) + s_1 \nu^-(g) &,\nu(g) + s_2 \nu^-(g))
\longmapsto (s_1 \nu^-(g), s_2 \nu^-(g)).
\end{align}
where we identify $\mathsf{T}_{\mathtt{N}(g)}\mathsf{U}\mathbb{A}$ with $\mathsf{T}_{{N}(g)}\mathbb{A} \times \mathsf{T}_{\nu(g)}\mathsf{S}^{1}$.

\subsection{Contraction Properties}

In this subsection we will prove that the leaves denoted by $\mathcal{L}^+$ contracts in the forward direction of the geodesic flow and the leaves denoted by $\mathcal{L}^-$ contracts in the backward direction of the geodesic flow. We will prove it only for the forward direction of the flow. The other case will follow similarly. We start with the following construction whose raison d'$\hat{\text{e}}$tre would be apparent in proposition \ref{lip}.

\begin{proposition}\label{prelip}
There exists a $\Gamma$-invariant map from $\mathsf{U}_{\hbox{\tiny $\mathrm{rec}$}}\mathbb{A}$ into the space of euclidean metrics on $\mathbb{R}^3\times\mathbb{R}^3$ sending $Z$ to $\|.\|_Z$ such that for all positive integer $n$, there exists a positive real number $t_n$ satisfying the following property: if $t>t_n$, $Z\in\mathsf{U}_{\hbox{\tiny $\mathrm{rec}$}}\mathbb{A}$ and $W\in\tilde{\mathcal{L}}^+_\text{Z}$ then
\begin{align*}
{\| i^+_{\tilde{\Phi}_t Z}(\tilde{\Phi}_t W) - i^+_{\tilde{\Phi}_t Z}(\tilde{\Phi}_t Z) \|}_{\tilde{\Phi}_t Z} \leqslant \frac{1}{2^n}{\| i^+_Z(W) - i^+_Z(Z)\|}_Z.
\end{align*}
\end{proposition}

\begin{proof}
Let $\langle | \rangle_{\mathtt{N}(g)}$ be a positive definite bilinear form on the tangent space $\mathsf{T}_{\mathtt{N}(g)}(\mathbb{A} \times \mathbb{V})$ satisfying the following properties,
\begin{align*}
&\text{1. }\langle (\nu^{\alpha}(g),0)| (\nu^{\beta}(g),0) \rangle_{\mathtt{N}(g)} = \langle (0,\nu^{\alpha}(g))| (0,\nu^{\beta}(g)) \rangle_{\mathtt{N}(g)} = {\delta}_{\alpha\beta},\\
&\text{2. }\langle (\nu^{\alpha}(g),0)| (0,\nu^{\beta}(g)) \rangle_{\mathtt{N}(g)} = \langle (0,\nu^{\alpha}(g))| (\nu^{\beta}(g),0) \rangle_{\mathtt{N}(g)} = 0.
\end{align*}
where ${\delta}_{\alpha\beta}$ is the dirac delta function with $\alpha,\beta$ in $\{., +, -\}$. We define the map $\|.\|$ as follows, 
\begin{align*}
\| X \|_{\mathtt{N}(g)} \defeq \sqrt{\langle X | X \rangle}_{\mathtt{N}(g)},
\end{align*}
where $X$ is in $\mathsf{T}_{\mathtt{N}(g)}(\mathbb{A} \times \mathbb{V})$.
Now from equation \ref{nu2}, equation \ref{limit2} and theorem \ref{N} we get that $\|.\|$ is $\Gamma$-invariant, that is,
\begin{align*}
\| \gamma X \|_{\gamma\mathtt{N}(g)} = \| X \|_{\mathtt{N}(g)}.
\end{align*}
Let $Z = \mathtt{N}(g)$ be in $\mathsf{U}_{\hbox{\tiny $\mathrm{rec}$}}\mathbb{A}$ and $W\in\tilde{\mathcal{L}}^+_\text{Z}$. Therefore there exists real numbers $s_1$ and $s_2$ such that 
\[W = ({N}(g) + s_1\nu^+(g) ,\nu(g) + s_2\nu^+(g) ).\] 
Hence the norm is
\begin{align} \label{15}
&{\| i^+_Z(W) - i^+_Z(Z)\|}_Z = {\| (s_1\nu^+(g), s_2\nu^+(g)) \|}_Z = \sqrt{s_1^2 + s_2^2}.
\end{align}
We note that $\tilde{\Phi}_t Z = ({N}(g) + t\nu(g),\nu(g))$ and using theorem \ref{N} we get that there exists a positive real number $t_1$ such that 
\[{N}(g) + t\nu(g)={N}(\tilde{\phi}_{t_1}g).\] 
Moreover $t$ and $t_1$ are related by the following formula,
\begin{align*}
t = \int\limits_{0}^{t_1} {f}(\tilde{\phi}_s g)ds.
\end{align*}
Therefore we have
\begin{align*}
\| i^+_{\tilde{\Phi}_t Z}(\tilde{\Phi}_t W) &- i^+_{\tilde{\Phi}_t Z}(\tilde{\Phi}_t Z)\|_{\tilde{\Phi}_t Z} = {\| ((s_1 + ts_2)\nu^+(g), s_2\nu^+(g)) \|}_{\tilde{\Phi}_t Z}\\
&= {\| ((s_1 + ts_2)\nu^+(g), s_2\nu^+(g)) \|}_{\mathtt{N}(\phi_{t_1}g)}\\
&= {\sqrt{( s_1 + ts_2 )^2 + s_2^2}} \text{ . } {\|(\nu^+(g),0) \|}_{\mathtt{N}(\phi_{t_1}g)}\\
&= {\sqrt{( s_1 + ts_2 )^2 + s_2^2}} \text{ . } {\| e^{-t_1}(\nu^+(\phi_{t_1}g),0) \|}_{\mathtt{N}(\phi_{t_1}g)}
\end{align*}
Hence the norm is
\begin{align}\label{16}
\| i^+_{\tilde{\Phi}_t Z}(\tilde{\Phi}_t W) - i^+_{\tilde{\Phi}_t Z}(\tilde{\Phi}_t Z)\|_{\tilde{\Phi}_t Z} \notag &= {\sqrt{( s_1 + ts_2 )^2 + s_2^2}}\text{ . }e^{-t_1}\\
&\leqslant \sqrt{2}\sqrt{s_1^2 + s_2^2}(1+t)e^{-t_1}.
\end{align}
We also know that ${\mathsf{U}_{\hbox{\tiny $\mathrm{rec}$}} \Sigma}$ is compact. Hence ${f}$ is bounded on $\mathsf{U}_{\hbox{\tiny $\mathrm{rec}$}}\mathbb{H}$. Therefore there exists a constant $c_1$ such that
\[t = \int\limits_{0}^{t_1} {f}(\tilde{\phi}_s g)ds \leqslant  \int\limits_{0}^{t_1} c_1 ds = c_1t_1.\]
We choose a constant $c$ bigger than $\max\{1,2c_1\}$ and get that
\begin{align}\label{17}
(1 + t) e^{-t_1} \leqslant c e^{-\frac{t}{2c_1}}.
\end{align}
Now by combining equation \ref{15}, inequalities \ref{16} and \ref{17} we get that 
\begin{align*}
{\| i^+_{\tilde{\Phi}_t Z}(\tilde{\Phi}_t W) - i^+_{\tilde{\Phi}_t Z}(\tilde{\Phi}_t Z) \|}_{\tilde{\Phi}_t Z} \leqslant \sqrt{2}c e^{-\frac{t}{2c_1}}{\| i^+_Z(W) - i^+_Z(Z)\|}_Z.
\end{align*}
Hence for all positive integer $n$, there exists $t_n\in\mathbb{R}$ such that if $t>t_n$, $Z\in\mathsf{U}_{\hbox{\tiny $\mathrm{rec}$}}\mathbb{A}$ and $W\in\mathcal{L}^+_\text{Z}$ then
\begin{align*}
{\| i^+_{\tilde{\Phi}_t Z}(\tilde{\Phi}_t W) - i^+_{\tilde{\Phi}_t Z}(\tilde{\Phi}_t Z) \|}_{\tilde{\Phi}_t Z} \leqslant \frac{1}{2^n}{\| i^+_Z(W) - i^+_Z(Z)\|}_Z.
\end{align*}
\end{proof}

\begin{proposition} \label{lip}
Let $d$ be a $\Gamma$-invariant distance on $\mathsf{U}_{\hbox{\tiny $\mathrm{rec}$}}\mathbb{A}$ which is locally bilipschitz equivalent to an euclidean distance and let $\|.\|$ be the $\Gamma$-invariant map from $\mathsf{U}_{\hbox{\tiny $\mathrm{rec}$}}\mathbb{A}$ to the space of euclidean metrics on $\mathbb{R}^3 \times \mathbb{R}^3$ as constructed in the proof of proposition \ref{prelip}. There exist positive constants $K$ and $\alpha$ such that for any $Z\in\mathsf{U}_{\hbox{\tiny $\mathrm{rec}$}}\mathbb{A}$ and for any $W\in\mathcal{L}^+_{Z}$, the following statements are true,
\begin{align*}
&\text{1. If  }d{(W,Z)} \leqslant \alpha, \text{ then }{\| i^+_Z(Z) - i^+_Z(W) \|}_{Z} \leqslant Kd{(W,Z)},\\
&\text{2. If  } {\| i^+_Z(Z) - i^+_Z(W) \|}_{Z} \leqslant \alpha, \text{ then } d{(W,Z)} \leqslant K{\| i^+_Z(Z) - i^+_Z(W) \|}_{Z}.
\end{align*}
\end{proposition}
\begin{proof}
Since $\Gamma$ acts cocompactly on $\mathsf{U}_{\hbox{\tiny $\mathrm{rec}$}}\mathbb{A}$ and both $d$ and $\|.\|$ are $\Gamma$-invariant, it suffices to prove the above assertion for $Z$ in a compact subset $D$ of $\mathsf{U}_{\hbox{\tiny $\mathrm{rec}$}}\mathbb{A}$, where $D$ is the closure of a suitably chosen fundamental domain.\\
We can define an euclidean distance $d_Z$ on $\mathsf{U}_{\hbox{\tiny $\mathrm{rec}$}}\mathbb{A}$, uniquely using the euclidean metric $\|.\|_Z$ on $\mathbb{R}^3 \times \mathbb{R}^3$, by taking the embedding of $\mathsf{U}_{\hbox{\tiny $\mathrm{rec}$}}\mathbb{A}$ in  $\mathbb{A} \times \mathbb{R}^3$. We notice that for any  $Z$ in $\mathsf{U}_{\hbox{\tiny $\mathrm{rec}$}}\mathbb{A}$ and for any $W$ in $\mathcal{L}^+_{Z}$, $d_Z(W,Z)$ is equal to $\|i^+_Z(W)-i^+_Z(Z)\|_Z$. Now, any two euclidean distances are bilipschitz equivalent with each other and by our hypothesis, $d$ is locally bilipschitz equivant to an euclidean distance. Therefore, in particular, $d$ is locally bilipschitz equivalent with $d_Z$ for $Z$ in $D$, that is, there exist constants $K_Z$ depending on $Z$, and open sets $U_Z$ around $Z$, such that the distance $d_Z$ and $d$ are $K_Z$ bilipschitz equivalent with each other on $U_Z$.\\
Let $C_{(X,Y)}$ for any $X$ and $Y$ in $D$, be a constant such that the distance $d_X$ and $d_Y$ are $C_{(X,Y)}$ bilipschitz equivalent with each other. It follows from the construction of the norm $\|.\|$, as done in proposition \ref{prelip}, that we can choose the constants $C_{(X,Y)}$ in such a way that $C_{(X,Y)}$ vary continuously on $(X,Y)$. As $D$ is compact it follows that $C_{(X,Y)}$ is bounded above by some constant $C$. Hence, for all $X$ and $Y$ in $D$, $d_X$ and $d_Y$ are $C$ bilipschitz equivalent with each other.\\
Now, we consider the open cover of $D$ by the open sets $U_Z$. As $D$ is compact, there exist points $Z_1, Z_2,..., Z_n$ in $D$ such that $U_{Z_1}, U_{Z_2},..., U_{Z_n}$ covers $D$. Let $\beta$ be the Lebesgue number of this cover for the distance $d$ and $K_0$ be the maximum of $K_{Z_1}, K_{Z_2},..., K_{Z_n}$. Therefore, for any $Z$ in $D$, the open ball of radius $\beta$ around $Z$ for the metric $d$, denoted by $B_d(Z,\beta)$, lies inside $U_{Z_j}$ for some $j$ in $\lbrace 1,2,...,n \rbrace$. Hence, $d$ and $d_{Z_j}$ are $K_0$ bilipschitz equivalent with each other on $B_d(Z,\beta)$. As $d_Z$ and $d_{Z_j}$ are $C$ bilipschitz equivalent with each other, it follows that $d$ and $d_Z$ are $CK_0$ bilipschitz equivalent with each other on $B_d(Z,\beta)$. Moreover, we note that the constants $\beta$, $C$, $K_0$ and hence also $CK_0$, does not depend on $Z$. Therefore, $d$ and $d_Z$ are $CK_0$ bilipschitz equivalent with each other on $B_d(Z,\beta)$, for all $Z$ in $D$.\\
As any two distances $d_X$ and $d_Y$, for all $X$, $Y$ in $D$ are $C$ bilipschitz equivalent with each other. Without loss of generality we can choose a point $X$ in $D$ and consider the distance $d_X$. The note that the set $\{B_d(Z,\beta): Z \in D\}$ is an open cover of $D$. Let $\beta_1$ be a Lebesgue number for this cover for the metric space $(D,d_X)$. Therefore, the open ball $B_{d_X}(Y_1,\beta_1)$ for any $Y_1$ in $D$, lies inside an open ball $B_d(Y_2,\beta)$ for some point $Y_2$ in $D$. Now, as $d$ and $d_Z$ are $CK_0$ bilipschitz equivalent with each other on the ball $B_d(Z,\beta)$ for all $Z$ in $D$, it follows that $d$ and $d_X$ are $CK_0$ bilipschitz equivalent with each other on the ball $B_{d_X}(Y_2,\beta_1)$. As $Y_2$ was chosen arbitrarily we have that $d$ and $d_X$ are $CK_0$ bilipschitz equivalent with each other on the ball $B_{d_X}(Y,\beta_1)$, for all $Y$ in $D$.\\
Now, we know that $d_X$ and $d_Z$ are $C$ bilipschitz equivalent with each other. Therefore we get that $d$ and $d_Z$ are $CK_0$ bilipschitz equivalent with each other on the ball $B_{d_Z}(Y,\frac{\beta_1}{C})$, for all $Y$ in $D$. In particular one has, $d$ and $d_Z$ are $CK_0$ bilipschitz equivalent with each other on the ball $B_{d_Z}(Z,\frac{\beta_1}{C})$. Finally, set $\alpha$ to be $\min \lbrace \frac{\beta_1}{C}, \beta\rbrace$ and $K$ to be $CK_0$ to get that for any $Z$ in $\mathsf{U}_{\hbox{\tiny $\mathrm{rec}$}}\mathbb{A}$ and $W$ in $\mathcal{L}^+_\text{Z}$ we have,
\begin{align*}
&\text{1. If  }d{(W,Z)} \leqslant \alpha, \text{ then }{\| i^+_Z(Z) - i^+_Z(W) \|}_{Z} \leqslant Kd{(W,Z)},\\
&\text{2. If  } {\| i^+_Z(Z) - i^+_Z(W) \|}_{Z} \leqslant \alpha, \text{ then } d{(W,Z)} \leqslant K{\| i^+_Z(Z) - i^+_Z(W) \|}_{Z}.
\end{align*}.
\end{proof}

\begin{theorem}
Let $\mathcal{L}^{\pm}$ be two laminations on $\mathsf{U}_{\hbox{\tiny $\mathrm{rec}$}}\mathbb{A}$ as defined in definitions \ref{lam1}, \ref{lam2} and let $\tilde{d}$ be the $\Gamma$ invariant metric, as defined in definition \ref{dist}. Under these assumptions, for the metric $\tilde{d}$ on $\mathsf{U}_{\hbox{\tiny $\mathrm{rec}$}}\mathbb{A}$ we have that
\begin{align*}
&\text{1. }\mathcal{L}^{+} \text{ is contracted in the forward direction of the geodesic flow, and, }\\
&\text{2. }\mathcal{L}^{-} \text{ is contracted in the backward direction of the geodesic flow.}
\end{align*}
\end{theorem}

\begin{proof}
Let $\|.\|$ be the $\Gamma$-invariant map from $\mathsf{U}_{\hbox{\tiny $\mathrm{rec}$}}\mathbb{A}$ to the space of euclidean metrics on $\mathbb{R}^3 \times \mathbb{R}^3$ as constructed in the proof of proposition \ref{prelip} and let $K$ and $\alpha$ be as in the proposition \ref{lip} for the distance $\tilde{d}$. We choose a positive integer $n$ such that 
\begin{align*}
\frac{K}{2^n} < 1 \text{ , } \frac{K^2}{2^n} \leqslant \frac{1}{2}.
\end{align*}
Let $t_n$ be the constant as in proposition \ref{prelip} for our chosen $n$. Also let $Z$ be in $\mathsf{U}_{\hbox{\tiny $\mathrm{rec}$}}\mathbb{A}$ and $W$ be in $\mathcal{L}^+_{Z}$, so that $\tilde{d}(W,Z) \leqslant \alpha$. Then using proposition \ref{lip} we get
\begin{align*}
{\| i^+_Z(W) - i^+_Z(Z) \|}_{Z} \leqslant K\tilde{d}{(W,Z)}.
\end{align*}
Furthermore, using proposition \ref{prelip} we get for all $t>t_n$ that
\begin{align*}
&{\| i^+_{\tilde{\Phi}_t Z}(\tilde{\Phi}_t W) - i^+_{\tilde{\Phi}_t Z}(\tilde{\Phi}_t Z) \|}_{\tilde{\Phi}_t Z} \leqslant \frac{1}{2^n}{\| i^+_Z(W) - i^+_Z(Z)\|}_Z.
\end{align*}
It follows that
\begin{align*}
{\| i^+_{\tilde{\Phi}_t Z}(\tilde{\Phi}_t W) - i^+_{\tilde{\Phi}_t Z}(\tilde{\Phi}_t Z) \|}_{\tilde{\Phi}_t Z} \leqslant \frac{K\alpha}{2^n} \leqslant \alpha.
\end{align*}
Hence again using proposition \ref{lip} we have
\begin{align*}
\tilde{d}(\tilde{\Phi}_t W, \tilde{\Phi}_t Z) \leqslant K {\| i^+_{\tilde{\Phi}_t Z}(\tilde{\Phi}_t W) - i^+_{\tilde{\Phi}_t Z}(\tilde{\Phi}_t Z) \|}_{\tilde{\Phi}_t Z}.
\end{align*}
Combining the above inequalities, for all $t>t_n$ we get
\begin{align} \label{convergence}
\tilde{d}(\tilde{\Phi}_t W, \tilde{\Phi}_t Z) &\leqslant \frac{K^2}{2^n} \tilde{d}{(W,Z)} \leqslant \frac{1}{2}\ \tilde{d}{(W,Z)}.
\end{align}
Hence $\mathcal{L}^+$ is contracted in the forward direction of the geodesic flow. The proof of the contraction of $\mathcal{L}^-$ follows similarly.
\end{proof}

\subsection{Metric Anosov structure on the quotient}

Let us now consider what happens in the quotient, that is, $\mathsf{U}_{\hbox{\tiny $\mathrm{rec}$}}\mathsf{M}$. Let $Z$ be in $\mathsf{U}_{\hbox{\tiny $\mathrm{rec}$}}\mathbb{A}$ and $\epsilon$ be a positive real number. We define, 
\begin{align*}
\mathcal{L}^{\pm}_\epsilon(Z) \defeq \mathcal{L}^{\pm}_Z \cap B_{\tilde{d}}(Z, \epsilon),
\end{align*}
and
\begin{align*}
\mathcal{K}_\epsilon(Z) \defeq {\Pi}_{Z} \left(\mathcal{L}^+_\epsilon(Z)\times \mathcal{L}^-_\epsilon(Z) \times (-\epsilon,\epsilon)\right) \subset \mathsf{U}_{\hbox{\tiny $\mathrm{rec}$}}\mathbb{A}
\end{align*}
where ${\Pi}_{Z}$ is the local product structure at $Z$ defined by the stable and unstable leaves.

We know that there exists a positive real number $\epsilon_0$ such that for any non identity element $\gamma$ of $\Gamma$ and for $Z$ in $\mathsf{U}_{\hbox{\tiny $\mathrm{rec}$}}\mathbb{A}$ we have,
\begin{align*}
\gamma (\mathcal{K}_{\epsilon_0}(Z)) \cap \mathcal{K}_{\epsilon_0}(Z) = \emptyset .
\end{align*}
\begin{proof}[Proof of Theorem \ref{mainthm1}]
Let us fix $\alpha$ as in proposition \ref{lip} and let $\epsilon_1$ be from the open interval $\left(0,\min\left\lbrace\alpha,\frac{\epsilon_0}{2}\right\rbrace\right)$. Now let $z$ be any point of $\mathsf{U}_{\hbox{\tiny $\mathrm{rec}$}}\mathsf{M}$ and let $Z$ be a point in $\mathsf{U}_{\hbox{\tiny $\mathrm{rec}$}}\mathbb{A}$ in the preimage of $z$. Our choice of $\epsilon_1$ gives us that the inequality \ref{convergence} holds for the geodesic flow on $\mathsf{U}_{\hbox{\tiny $\mathrm{rec}$}}\mathbb{A}$ for the points in the chart $\mathcal{K}_{\epsilon_1}(Z)$. Hence the inequality \ref{convergence} also holds for the geodesic flow on $\mathsf{U}_{\hbox{\tiny $\mathrm{rec}$}}\mathsf{M}$ for points in the chart which is in the projection of $\mathcal{K}_{\epsilon_1}(Z)$.

Therefore $\underline{\mathcal{L}}^+$, the projection of $\mathcal{L}^{+}$ in $\mathsf{U}_{\hbox{\tiny $\mathrm{rec}$}}\mathsf{M}$, is contracted in the forward direction of the geodesic flow on $\mathsf{U}_{\hbox{\tiny $\mathrm{rec}$}}\mathsf{M}$. A similar proof holds for $\underline{\mathcal{L}}^-$, the projection of $\mathcal{L}^{-}$ in $\mathsf{U}_{\hbox{\tiny $\mathrm{rec}$}}\mathsf{M}$, too.
\end{proof}

\section{Anosov representations}

In this section we define the notion of an Anosov representation in the context of the non-semisimple Lie group $\mathsf{G} \defeq \mathsf{SO}^0(2,1)\ltimes\mathbb{R}^3$. 

\subsection{Pseudo-Parabolic subgroups}

Let $\mathbb{X}$ be the space of all affine null planes. We observe that $\mathsf{G}$ acts transitively on $\mathbb{X}$. Hence for all $P\in\mathbb{X}$ we have
\[\mathbb{X} = \mathsf{G}.P \cong \mathsf{G}/\mathsf{Stab}_{\mathsf{G}}(P).\]
\begin{definition}\label{levi}
If $P\in\mathbb{X}$ then we define
\[\mathsf{P}_P\defeq\mathsf{Stab}_{\mathsf{G}}(P).\]
We call $\mathsf{P}_P$ a $\textit{pseudo-parabolic}$ subgroup of $\mathsf{G}$.
\end{definition}
Let $\mathsf{V}(P)$ denote the vector space underlying a null plane $P$, let $v_0\defeq(1,0,0)^t$ and $v_0^\pm\defeq(0,\pm1,1)^t$ and let $\mathcal{C}$ be the upper half of $\mathsf{S}^0\backslash\{0\}$. Now we consider the space
\[\mathcal{N}\defeq\{(P_1,P_2)\mid P_1,P_2\in\mathbb{X}, \mathsf{V}(P_1)\neq \mathsf{V}(P_2)\}\]
and define the following map
\begin{align*}
v : \mathcal{N} &\longrightarrow \mathsf{S}^{1}\\
(P_1,P_2) &\longmapsto v(P_1,P_2)
\end{align*}
where $v(P_1,P_2)\in\mathsf{V}(P_1)\cap\mathsf{V}(P_2)\cap\mathsf{S}^{1}$ is such that if $v(Q_1)\in\mathsf{V}(Q_1)\cap\mathcal{C}$ and $v(Q_2)\in\mathsf{V}(Q_2)\cap\mathcal{C}$ then $(v(Q_1),v(Q_1,Q_2),v(Q_2))$ gives the same orientation as $(v_0^+,v_0,v_0^-)$. We observe that
\[v(P_1,P_2)=-v(P_2,P_1).\]

\begin{proposition}\label{open}
The space $\mathcal{N}$ is the unique open $\mathsf{G}$ orbit in $\mathbb{X}\times\mathbb{X}$ for the diagonal action of $\mathsf{G}$ on $\mathbb{X}\times\mathbb{X}$.
\end{proposition}
\begin{proof}
We start by observing that $\mathcal{N}$ is open and dense in $\mathbb{X}\times\mathbb{X}$.
Now let $(P_1,P_2)$ and $(Q_1,Q_2)$ be two arbitrary points in $\mathcal{N}$. We consider the vector $v(P_1,P_2)\in\mathsf{S}^{1}$ corresponding to the point $(P_1,P_2)$ and the vector $v(Q_1,Q_2)\in\mathsf{S}^{1}$ corresponding to the point $(Q_1,Q_2)$. Now as $\mathsf{SO}^0(2,1)$ acts transitively on $\mathsf{S}^{1}$ we get that there exist $g\in\mathsf{SO}^0(2,1)$ such that 
\[v(Q_1,Q_2)=g.v(P_1,P_2).\]
We choose $X(Q_1,Q_2)\in Q_1\cap Q_2$ and $X(P_1,P_2)\in P_1\cap P_2$ and observe that 
\[(e,X(Q_1,Q_2)-O)\circ(g,0)\circ(e,X(P_1,P_2)-O)^{-1}.P_1=Q_1,\]
\[(e,X(Q_1,Q_2)-O)\circ(g,0)\circ(e,X(P_1,P_2)-O)^{-1}.P_2=Q_2,\]
where $e$ is the identity element in $\mathsf{SO}^0(2,1)$. Therefore $\mathcal{N}$ is an open $\mathsf{G}$ orbit in $\mathbb{X}\times\mathbb{X}$. Now as $\mathbb{X}\times\mathbb{X}$ is connected the result follows.
\end{proof}

Let $\mathsf{N}$ be the space of oriented space like affine lines. We think of $\mathsf{N}$ as the space $\mathsf{U}\mathbb{A}/\sim$ where $(X,v)\sim(X_1,v_1)$ if and only if $(X_1,v_1)=\tilde{\Phi}_t(X,v)$ for some $t\in\mathbb{R}$. We denote the equivalence class of $(X,v)$ by $[(X,v)]$. Now let us consider the following map
\begin{align*}
\imath^\prime : \mathcal{N} &\longrightarrow \mathsf{N}\\
(P_1,P_2) &\longmapsto [(X(P_1,P_2),v(P_1,P_2))]
\end{align*}
where $X(P_1,P_2)$ is any point in $P_1\cap P_2$. We observe that $\imath^\prime$ gives a $\mathsf{G}$ equivariant map.

Let us denote the plane passing through $X$ with underlying vector space generated by the vectors $w_1$ and $w_2$ by $P_{X,w_1,w_2}$. Now we consider another map
\begin{align*}
\imath : \mathsf{U}\mathbb{A} &\longrightarrow \mathcal{N}\\
(X,v) &\longmapsto (P_{X,v,v^+},P_{X,v,v^-})
\end{align*}
where $v^\pm\in\mathcal{C}$ such that $\langle v^\pm\mid v\rangle = 0$ and $(v^+,v,v^-)$ gives the same orientation as $(v_0^+,v_0,v_0^-)$. We observe that $\imath$ is a $\mathsf{G}$ equivariant map. Now as $P_{X+tv,v,v^+}=P_{X,v,v^+}$ and $P_{X+tv,v,v^-}=P_{X,v,v^-}$ we get that the map $\imath$ gives rise to a map, which we again denote by $\imath$,
\[\imath :\mathsf{N} \longrightarrow \mathcal{N}.\] 
Moreover, we observe that $\imath\circ\imath^\prime=\mathsf{Id}$ and $\imath^\prime\circ\imath=\mathsf{Id}$.

\subsection{Geometric Anosov structure}

Geometric Anosov structures were first intoduced by Labourie in \cite{orilab}. In this subsection we give an appropriate definition of geometric Anosov property and show that $(\mathsf{U}_{\hbox{\tiny $\mathrm{rec}$}}\mathsf{M},\mathcal{L})$ admits a geometric Anosov structure.

Let $(P^+,P^-)\in\mathcal{N}$ such that $P^+\defeq P_{O,v_0,v_0^+}$ and $P^-\defeq P_{O,v_0,v_0^-}$. We denote $\mathsf{Stab}_\mathsf{G}(P^\pm)$ respectively by $\mathsf{P}^\pm$.
We note that the pair $\mathbb{X}^\pm\defeq\mathsf{G}/\mathsf{P}^\pm$ gives a pair of continuous foliations on the space $\mathsf{N}$ whose tangential distributions $\mathsf{E}^\pm$ satisfy
\[\mathsf{TN} = \mathsf{E}^+\oplus\mathsf{E}^-.\]
\begin{definition}
We say that a vector bundle $\mathsf{E}$ over a compact topological space whose total space is equipped with a flow $\{\varphi_t\}_{t\in\mathbb{R}}$ of bundle automorphisms is $\textit{contracted}$ by the flow as $t\to\infty$ if for any metric $\|.\|$ on $\mathsf{E}$, there exists positive constants $t_0$, $A$ and $c$ such that for all $t>t_0$ and for all $v$ in $\mathsf{E}$ we have
\begin{align*}
\|\varphi_t(v)\| \leqslant Ae^{-ct}\|v\|.
\end{align*}
\end{definition}
\begin{definition}
Let $\mathcal{L}$ denote the orbit foliation of $\mathsf{U}_{\hbox{\tiny $\mathrm{rec}$}}\mathsf{M}$ under the flow $\Phi$.
We say that $(\mathsf{U}_{\hbox{\tiny $\mathrm{rec}$}}\mathsf{M},\mathcal{L})$ admits a $\textit{geometric}$ $(\mathsf{N}, \mathsf{G})$-$\textit{Anosov structure}$ if there exist a map
\[F: \widetilde{\mathsf{U}_{\hbox{\tiny $\mathrm{rec}$}}\mathsf{M}}\longrightarrow\mathsf{N}\]
such that the following holds:
\begin{enumerate}
\item For all $\gamma\in\Gamma$ we have $F\circ\gamma=\gamma\circ F$,
\item For all $t\in\mathbb{R}$ we have $F\circ\tilde{\Phi}_t=F$,
\item By the flow invariance, the bundles $F^\pm\defeq F^*\mathsf{E}^\pm$ are equipped with a parallel transport along the orbits of $\tilde{\Phi}$. The bundle $F^+$ gets contracted by the lift of the flow $\tilde{\Phi}_t$ as $t\to\infty$ and $F^-$ gets contracted by the lift of the flow $\tilde{\Phi}_t$ as $t\to-\infty$.
\end{enumerate} 
\end{definition}

\begin{proof}[Proof of Theorem \ref{geomano}]
Let us define the map $F$ as follows:
\begin{align*}
F: \widetilde{\mathsf{U}_{\hbox{\tiny $\mathrm{rec}$}}\mathsf{M}} &\longrightarrow\mathsf{N}\\
(X,v) &\longmapsto [(X,v)]
\end{align*}
We note that the map $F$ is clearly $\Gamma$-equivariant and is also invariant under the flow $\tilde{\Phi}$. Now we observe that 
\[\mathsf{T}_{\imath\left([(X,v)]\right)}\mathsf{G}/\mathsf{P}^-\cong\mathbb{R}.v^+\oplus\mathbb{R}.v^+\]
and 
\[\mathsf{T}_{\imath\left([(X,v)]\right)}\mathsf{G}/\mathsf{P}^+\cong\mathbb{R}.v^-\oplus\mathbb{R}.v^-\]
where $v^+,v^-\in\mathcal{C}$ such that $\langle v^\pm\mid v\rangle=0$ and $(v^+,v,v^-)$ gives the same orientation as $(v_0^+,v_0,v_0^-)$.

Now using proposition \ref{prelip} we notice that $F^+$ gets contracted by the lift of the flow $\tilde{\Phi}_t$ as $t\to\infty$ and $F^-$ gets contracted by the lift of the flow $\tilde{\Phi}_t$ as $t\to-\infty$. Moreover, as $\mathsf{U}_{\hbox{\tiny $\mathrm{rec}$}}\mathsf{M}$ is compact we have that the convergence is independent of the choice of the metric.
\end{proof}


\begin{thebibliography}{99}


\bibitem{pressure metric} BRIDGEMAN, M., CANARY, R., LABOURIE, F. and SAMBARINO, A.: \emph{``The pressure metric for convex homomorphisms."} arXiv preprint arXiv:1301.7459 (2013).
\bibitem{jones} CHARETTE, V., GOLDMAN, W. M. and JONES, C. A.: \emph{``Recurrent Geodesics in Flat Lorentz 3-Manifolds."} (2001).
\bibitem{dgk} DANCIGER, J, GUERITAUD, F. and KASSEL, F.: \emph{``Geometry and Topology of Complete Lorentz SpaceTimes of Constant Curvature."} (2013).
\bibitem{D} DRUMM, T.: \emph{``Fundamental polyhedra for Margulis space-times"}, Doctoral Dissertation, University of Maryland (1990).
\bibitem{fried} FRIED, D. and GOLDMAN, W. M.: \emph{``Three-dimensional affine crystallographic groups."} Adv. in Math. 47, 1-49, (1983).
\bibitem{geodesic} GOLDMAN, W. M. and LABOURIE, F.: \emph{``Geodesics in Margulis spacetimes."} Ergodic Theory and Dynamical Systems, 32, 643-651, (2012).
\bibitem{labourie invariant} GOLDMAN, W. M., LABOURIE, F. and MARGULIS, G.: \emph{``Proper affine actions and geodesic flows of hyperbolic surfaces."} Annals of mathematics 170.3, 1051-1083, (2009).
\bibitem{guiwien}  GUICHARD, O. and WIENHARD, A.: \emph{``Anosov representations: Domains of discontinuity and applications."} Inventiones Mathematicae, Volume 190, Issue 2, 357-438, (2012).
\bibitem{orilab} LABOURIE, F.: \emph{``Anosov Flows, Surface Groups and Curves in Projective Space."} Inventiones Mathematicae, Volume 165, Issue 1, 51-114, (2006).
\bibitem{marg1} MARGULIS, G. A.: \emph{``Free completely discontinuous groups of affine transformations."} Dokl. Akad. Nauk SSSR 272, 785-788, (1983).
\bibitem{marg2} MARGULIS, G. A.: \emph{``Complete affine locally flat manifolds with a free fundamental group."} J. Soviet Math. 134, 129-134, (1987).
\bibitem{mess} MESS, G.: \emph{``Lorentz spacetimes of constant curvature."} Geom. Dedicata 126, 3-45, (2007).
\bibitem{milnor} MILNOR.J: \emph{``On fundamental groups of complete affinely flat manifolds."} Adv. in Math.25, 178-187, (1977).

\end{thebibliography}
\end{document}